\documentclass[11pt, oneside, a4paper]{amsart}

\newcommand{\sprodA}[2]{\langle #1  ,#2 \rangle_A }
\newcommand{\sprod}[2]{\langle #1  ,#2 \rangle}

\def\dvA{\text{div}_A}
\newcommand{\modA}[1]{\left| #1\right|_A}

\def\lotw{l.o.t(s^3\lambda^4\varphi^3 |w|^2)}
\def\lotnw{l.o.t.(s\lambda^2\varphi|\nabla w|^2)}
\def\lotwnw{l.o.t(s^3\lambda^4\varphi^3 |w|^2+s\lambda^2\varphi|\nabla w|^2)}

\newcommand{\DeltaA}[1]{\Delta_A #1}
\newcommand{\dnuA}[1]{\dfrac{\partial #1}{\partial\nu_A} }

\def\ra{\rightarrow}

\usepackage{amsfonts, amsmath, amssymb, amsthm}
\newtheorem{theorem}{Theorem}
\newtheorem{proposition}{Proposition}

\newtheorem{remark}{Remark}
\newtheorem{lemma}{Lemma}

\newtheorem{comments}{Comments}

\newcommand{\R}{\mathbb{R}}
\newcommand{\A}{\mathcal{A}}

\newcommand{\g}{\overline{g}}
\newcommand{\w}{\tilde{w}}

\numberwithin{equation}{section}
\newcommand{\fin}{\hfill\rule{2mm}{2mm}\medskip }
\usepackage[left=1.25in, right=0.75in, top=1.25in, bottom=1in, includefoot, headheight=12pt]{geometry}

\subjclass[2010]{35R30, 35K10, 35K57, 	93B07 }

\keywords{Parabolic systems, , Reaction-diffusion systems, Carleman estimates, Inverse problems, Source estimates.}

 \thanks{$^{*}$Faculty of Mathematics, Al.I.Cuza University and Octav Mayer Institute of Mathematics,   Ia\c{s}i, Romania, e-mail: catalin.lefter@uaic.ro. Supported by a grant of the Ministry of Research,
 	Innovation and Digitization, CNCS - UEFISCDI,
 		project number PN-III-P4-PCE-2021-0006, within PNCDI III.}

 \thanks{$^{**}$Faculty of Mathematics, Al.I.Cuza University and Octav Mayer Institute of Mathematics,   Ia\c{s}i, Romania, e-mail: alexandra.melnig@uaic.ro. Supported by European Social Fund, through Operational Programme Human Capital 2014-2020, project number POCU/993/6/13/153322, project title "Educational and training support for PhD students and young
 researchers in preparation for insertion into the labor market". }
\begin{document}
\title[Source stability for reaction-diffusion systems]{Reaction-diffusion systems in annular domains: source stability estimates with  boundary observations}
\maketitle
{\footnotesize
	
	\centerline{\scshape C\u at\u alin-George Lefter$^*$, Elena-Alexandra Melnig$^{**}$}
	\medskip
	{\footnotesize
		
	}}

\begin{abstract}

{We consider systems of reaction-diffusion  equations coupled in zero  order terms, with general homogeneous boundary conditions in domains with a particular geometry (annular type domains). We establish Lipschitz stability estimates in $L^2$ norm   for the source in terms of the solution and/or its normal derivative on a connected component of the boundary. The main tools are represented by: appropriate Carleman estimates in $L^2$ norms, with boundary observations, and positivity improving properties for the solutions to parabolic equations and systems. }

\end{abstract}

\maketitle
\date{}

\section*{Introduction}
 In this paper we consider systems of linear and semilinear parabolic equations, coupled in zero-order terms, with general homogeneous boundary conditions  for each equation.  We are interested in obtaining stability estimates for the inverse problem of recovering the source of the system in terms of the solution measured on a part of the boundary. 
 

 We mention  the paper of O.Yu. Imanuvilov and M. Yamamoto, \cite{imayam1998}, as a starting point for our work; in this paper the authors consider linear parabolic problems and obtain    $L^2$ estimates for  the source  from a particular class of sources and with  observations on the solution either in a subdomain or on the boundary.   In their case the estimates involve norms of the solution and its time derivative for the case of internal observations or, in the case of boundary observations,  norms of the solution and its' gradient  and time derivative.
 
 For the case of internal observation more manageable classes of sources were considered in \cite{ladyA1} and \cite{ladyA2} and, with different approach, by using appropriate Carleman inequalities, stability estimates are obtained without measurements on the time derivative of the solution. In addition, in  \cite{ladyA2}, it is considered that the systems model reaction-diffusion processes and thus positivity of sources and solutions is also exploited in either linear or semilinear case.
 
  The systems of parabolic equations we consider here are also intended to model reaction-diffusion phenomena, with homogeneous  boundary conditions imposed to the components of the system, which may be Dirichlet, Neumann or Robin in arbitrary combination.
  The  observations made on the system are on the boundary and the respective quantities appearing in the stability estimates we obtain involve the solution and/or the conormal derivatives of the components of the solution, but no tangential component of the gradient or time derivative of the solution are needed.
  
  The classes of sources are further improved and we mention here the two main tools we use: First we obtain an appropriate Carleman estimate which, for technical reasons imposed by the need of particular auxiliary functions, require a special geometry of the domain that we call here of annular type.
  The second tool is represented by strong maximum type principles and invariance results for convex sets under the action of the parabolic flows defined by variational solutions. More precisely, we need positivity improving properties in an unique continuation result for positive variational solutions corresponding to positive sources in either linear case or in the semilinear models. Such facts occur under  sign conditions on the nonlinearities and  on the zero order coefficients representing the zero order coupling terms.

\section{Preliminaries and main results}

We consider  semilinear parabolic systems, with couplings involving only zero order terms, of the form: 
\begin{equation}\label{rd-sys}
\left\lbrace
\begin{array}{ll}
D_ty_i-\sum\limits_{j,k=1}^N D_j(a_i^{jk} D_ky_i)+\sum\limits_{k=1}^N b_i^k D_ky_i+f_i(t,x,y_1,...,y_n)=g_i &(0,T)\times\Omega,
\\
\displaystyle\beta_i(x)\frac{\partial y_i}{\partial \nu_{A_i}}+\eta_i(x)y_i=0  &(0,T)\times\partial\Omega  \\

\end{array} i\in\overline{1,n}.
\right.\end{equation} 
Here  $T>0$ and $\Omega\subset\R^N$ is  a bounded  connected domain with smooth boundary,   $\partial\Omega$ of class $C^2$.
Denote by $Q=(0,T)\times\Omega$ and $\partial_PQ=\{0\}\times\overline{\Omega}\cup[0,T]\times\partial\Omega$ the so called parabolic boundary of $Q$.
For the ease of notation, whenever we will refer to vector valued functions with elements in a given Lebesgue or Sobolev space, we will write, when no confusion is possible, $L^2(\Omega)$ instead of $ [L^2(\Omega)]^n$, $H^1(\Omega)$ instead of $[H^1(\Omega)]^n$, etc. Denote by $\langle\cdot,\cdot\rangle$ the canonical scalar product in $\R^N$ and by $\langle\cdot,\cdot\rangle_W$ the scalar product in a given Hilbert space $W$.

\medskip

We describe first the framework of the problem and the hypotheses:

\subsubsection*{\textbf{\rm (H1) }Geometry of the domain}

The boundary of the domain has the structure $\partial\Omega=\Gamma_0\cup\Gamma_1$, with $\Gamma_0,\Gamma_1$ connected components of $\partial\Omega$, and there exists $\psi_0\in C^2(\overline{\Omega},[k_0,k_1])$  for some $k_0<k_1$ with 
\begin{equation}\label{psi0}
\psi_0|_{\Gamma_0}=k_0,\,\psi_0|_{\Gamma_1}=k_1,\,|\nabla\psi_0|\not=0\text{ in } \overline\Omega.
\end{equation}

\begin{remark}

Essentially, such a function $\psi_0$ exists if  $\Gamma_0$ is a smooth deformation retract of $\overline\Omega$. More precisely, by adding some technical assumptions to what is usually named a deformation retract, such a $\psi_0$ exists if there is a smooth function $\Phi:\overline\Omega\times[0,1]\rightarrow\overline\Omega$ with the properties:
\begin{itemize}
	\item[-] $\Phi(\cdot,0)=\text{Id}$, $\Phi(\cdot,s)|_{\Gamma_0}=\text{Id}$, $s\in[0,1]$;
		\item[-] $\text{Range }\Phi(\cdot,1)=\Gamma_0$;
	\item[-] $\Phi(\cdot,s)$ is a diffeomorphism from $\overline \Omega$ onto its image for all $s\in[0,1)$;
	\item[-] $\Phi(\cdot,s)$ has strictly decreasing range in the sense that: if $s_1<s_2$ then $\Phi(\overline\Omega,s_2)\subset\Phi(\Omega,s_1)$ and
	$$
	\left|\frac{\partial \Phi}{\partial s}(x,s)\right|\not=0,\quad x\in \overline\Omega,s\in[0,1].
		$$

	.

\end{itemize} 
With such a $\Phi$, the auxiliary function  $\psi_0$ may be defined as 
$$
\psi_0(x)=1-s \quad\text{  for }\quad x\in\Phi(\Gamma_1,s),
$$ 
with $k_0=0,k_1=1.$

 Conversely, given $\psi_0$ with the required properties one may construct a deformation $\Phi$ as above. Without entering into details we describe this: consider first the vector field on $\overline\Omega$
 $$
 F(x)=\frac{\nabla\psi_0(x)}{|\nabla\psi_0(x)|^2}
 $$
and extend it smoothly and bounded to $\R^N$. Consider the flow associated to $F$:
$$
\frac{d}{ds}\gamma_x(s)=F(\gamma_x(s)), \, \gamma_x(0)=x.
$$
Observe that if $x\in\Gamma_0$, then $\gamma_x(s)\in\overline\Omega$ for $s\in[0,1]$ and $\psi(\gamma_x(s))=s$.
The domain $\overline\Omega$ is then filled with the curves $\gamma_x([0,1])$, $x\in\Gamma_0$. The deformation $\Phi$ may be defined along these curves: if $y\in\overline\Omega$,$y=\gamma_x(\tau)$ for some $x\in\Gamma_0$ and $\tau\in[0,1]$, then we define
$$
\Phi(y,s)=\gamma_x(\tau-s\tau),\,s\in[0,1].
$$
For our results the simplest admissible domain is a domain which is diffeomorphic to an annulus $\overline {B_2(0)\setminus B_1(0)}$. However, one may imagine other domains, for example the domain in $\R^3$ limited by two coaxial tori.

This geometry of the domains, that we call of annulus type,  will be essential for our Carleman observability estimates with boundary observations.
\end{remark}

\subsubsection*{\textbf{\rm (H2) }Operators and boundary conditions}
\begin{itemize}
\item[-]  $a_i^{jk}\in W^{1,\infty}(\Omega) $, $b_i^k\in L^\infty(\Omega)$ and $a_i^{jk}$
 satisfy the ellipticity condition: $$\exists \mu>0 \text{ \textit{s.t.} }\sum_{j,k=1}^N a_i^{jk}(x)\xi_j\xi_k\geq\mu|\xi|^2,\quad\forall\xi\in\R^N,\quad(t,x)\in Q,  i=\overline{1,n},$$
\item[-] For $i=\overline{1,n}$ the coefficients $\beta_i\in C(\partial\Omega,\{0,1\})$, $\eta_i\in L^\infty(\partial\Omega), \,  \eta_i\geq 0 \text{ and }\beta_i+\eta_i>0$ in $\partial\Omega$;
\item[-] Denoting by $A_i=A_i(x)=(a_i^{jk}(x))_{jk}$ the matrix valued functions constructed with the coefficients in the principal part and by $\nu$ the exterior normal vector field to the boundary $\partial\Omega$, the conormal derivative of a function $z\in H^1(\Omega)$ is
$$
\frac{\partial}{\partial\nu_{A_i}}z=\langle A_i(x)\nabla z(x),\nu(x)\rangle_{\R^N},\, x\in\partial\Omega.
$$
\end{itemize}

\subsubsection*{\textbf{\rm (H3) }Nonlinearity}
The coupling functions (representing reaction terms in a reaction-diffusion model) satisfy $f_i:\R\times\overline\Omega\times\R^n\longrightarrow\R$ are $C^1$ smooth, $f_i(t,x,\textbf{0})\equiv 0$ and one of the following hypotheses hold: either
\begin{itemize}
\item[-]   $f_i(t,x,y_1,\ldots,y_{i-1},0,y_{i+1},\dots, y_n)=0,  i =\overline{1,n},   t>0,x\in\overline\Omega,\, \textbf{y}\ge0;$

  or
\item[-]  $f_i(t,x,y_1,\ldots,y_{i-1},0,y_{i+1},\dots, y_n)\leq 0,  i =\overline{1,n},   t>0,x\in\overline\Omega,\,\textbf{ y}\ge0$
and
$$
\frac{\partial}{\partial y_l}f_i(t,x,\textbf{y})\leq0, i,l =\overline{1,n}, i\ne l,  t>0,x\in\overline\Omega,\, \textbf{y}\ge0.
$$
\end{itemize}
\begin{remark}
Less regularity with respect to $t,x$ may be assumed for the coupling functions $f_i$; we need essentially that  
$$
\frac{\partial f_i}{\partial y_i}(t,x, y_1(t,x), \ldots,y_n(t,x))\in L^\infty(Q),\,\forall y_1,\ldots,y_n\in L^\infty(Q).
$$
\end{remark}
\subsubsection*{\textbf{\rm (H4) }Observation operators}
We consider the boundary observation  operator: 
$$
\zeta(\textbf{y})=(\zeta_i(\textbf{y}))_i,
$$
$$\zeta_i:[H^1(\Omega)]^n\rightarrow L^2(\Gamma_1),\, \zeta_i=\zeta_i(\textbf{y})=\gamma_i(x)\frac{\partial y_i}{\partial \nu_{A_i}}+\delta_i(x)y_i,  x\in\Gamma_1, \textbf{y}=(y_1,\ldots,y_n).$$
 with $\gamma_i,\delta_i\in L^\infty(\partial\Omega)$, $i=\overline{1,n}$. We assume the following compatibilty condition between the coefficients appearing in the boundary conditions and the coefficients appearing in the observation operators: for some $\epsilon>0$
 $$
 \det \left(\begin{matrix}
 	\gamma_i&\delta_i\\
 	\beta_i& \eta_i
 \end{matrix}\right)=\gamma_i\eta_i-\beta_i\delta_i\ge\epsilon>0 \text{ in } \Gamma_1.
 $$

 \subsubsection*{\textbf{\rm (H5) }Classes of sources and admissible solutions}
 The following classes  of sources are considered in the paper: for   $k>0 $,
 \begin{equation}\label{ch5GsetD1}
 \mathcal{G}_{k}=\left\lbrace \textbf{g}\in L^2(0,T;[L^2(\Omega)]^n):\,\textbf{g}\ge 0 \text{ s.t. } \|\textbf{g}\|_{L^2(Q)}\leq k \|\textbf{g}\|_{L^1(Q)}\right\rbrace
 \end{equation}
As the parabolic systems we consider model reaction-diffusion processes we are working in the following classes of solutions: for $M>0$ let

\begin{equation}\label{ch5Fset}
\mathcal{F}_{M}=\lbrace\textbf{ y}\in [W^{2,1}(Q)]^n\cap [L^\infty(Q)]^n:\textbf{ y}\geq 0, \, \|\textbf{y}\|_{[L^\infty(Q)]^n}\leq M \rbrace .
\end{equation}
  where $W^{2,1}(Q)=L^2(0,T;H^2(\Omega))\cap W^{1,2}(0,T;L^2(\Omega))$

  The main result in the paper giving boundary $L^2$ source stability estimates  is:
  \begin{theorem}\label{thnonlinear}
   	
  	Let $ M,k>0$, and assume that  the sources in \eqref{rd-sys} belong to  $\mathcal{G}_{k}$ and the associated solutions  $\textbf{y}\in\mathcal{F}_{M}$.  Under hypotheses {\rm (H1)-(H5)} there exists  $C=C(\Omega,M,k)>0$ such that 
  	\begin{equation}\label{estgLq2}
  		\begin{aligned}
  			&\left\|\textbf{g}\right\|_{L^2(Q)}\leq C\|\zeta(\textbf{y})\|_{L^2(\Sigma_1)}.
  		\end{aligned}
  	\end{equation}
  \end{theorem}

  The approach for obtaining source estimates for nonlinear systems is passing through the study of the following  linear problem:
   \begin{equation}\label{sysinitial}
   \left\lbrace
   \begin{array}{ll}
   D_ty_i-\sum\limits_{j,k=1}^N D_j(a_i^{jk} D_ky_i)+\sum\limits_{{\substack{k=\overline{1,N}}}}b_i^{k} (x)D_ky_i+\sum\limits_{l=\overline{1,n}}c_{il}(t,x) y_l=g_i, &(0,T)\times\Omega,
   \\
   \beta_i(x)\frac{\partial y_i}{\partial \nu_{A_i}}+\eta_i(x)y_i=0 ,  &(0,T)\times\partial\Omega, 
   \end{array}
   \right.\end{equation} 
   under the following hypotheses:
   \begin{itemize}
   \item[-] The coefficients in the principal part $a_i^{jk}$ and the coefficients in the boundary conditions $\beta_i, \eta_i$ are as in (H2);
   \item[-] The lower-order operators are given by:
     \begin{equation}\label{secpart} L^1_iw=\sum\limits_{{\substack{k=\overline{1,N}}}}b_i^{k} (x)D_kw,\quad L^0_i(t)\textbf{y}=\sum\limits_{l=\overline{1,n}}c_{il}(t,x) y_l, \quad 
     \end{equation}
     with  coefficients  $\textbf{b}_i=(b_i^1,\ldots,b_i^N)^\top,b^{k}_i\in L^\infty(\Omega), c_{il}\in L^\infty(Q), i,l=\overline{1,n}, k=\overline{1,N}$. (Here $w$   denotes some scalar function while $y$ is a vector valued function) The system is coupled in the zero order terms through the operators $L^0_i.$
   \end{itemize}
   
   The inverse source estimates with boundary observations for the linear case is given in the following theorem:
  
   \begin{theorem}\label{thlinear}
Assume that  the sources $\textbf{g}=(g_i)_i, g_i\ge0, g_i\in L^2(0,T;L^2(\Omega)), i=\overline{1,n}$  belong to  $\mathcal{G}_k$ for some $k>0$ and
the coupling  coeffiecients $c_{il}\in L^\infty(Q)$ satisfy the sign condition  
$$c_{il}\leq 0,  \quad i=\overline{1,n}, i\neq l.$$
 Then, the sources $\textbf{g}\in\mathcal{G}_{k}$ and  corresponding variational solutions $y$ to \eqref{sysinitial} satisfy the source estimate with boundary observation  
	 \begin{equation}\label{estginflin}
	\left\|\textbf{g}\right\|_{L^2(Q)}\leq C\|\zeta(\textbf{y})\|_{L^2(\Sigma_1)},
	\end{equation} 
	for some constant $C=C(k,\|c_{il}\|_{L^\infty(Q)})>0$.
\end{theorem}
Before proving Theorem \ref{thlinear} and Theorem \ref{thnonlinear} in \S\ref{sec:stability}, we study first the main tools we need in our approach.
First, in \S\ref{sec:var} we present the variational framework for initial-boundary value problems associated to parabolic equations or systems and focus on the posivity and positivity improving effects of the parabolic flows.
Then, in \S\ref{sec:Carleman} we prove a special Carleman estimate for parabolic problems with homogeneous boundary conditions and with boundary observations involving conormal derivatives and/or trace of the solution, without need of either tangential component of the gradient of the solution or time derivative of the solution.

 \section{Variational solutions to parabolic equations and systems. Positivity of solutions and positivity improving properties.}
 \label{sec:var}

 \subsubsection*{Abstract parabolic problems.}

 Let $V,H$ be two Hilbert spaces, $V\subset H$ with continuous and dense embedding. This defines the Gelfand triple  $$V\subset H \subset V^*,$$
 where
  $V^*$ is the topological dual of $V$ with pivot space $H$. Denoting by $(\cdot,\cdot)_{V,V^*}$ the duality between $V$ and $V^*$ one has $(v,h)_{V,V^*}=\langle v,h\rangle_H$ when $v\in V$ and $h\in H$.

  Consider a bilinear form 
 $$a:V\times V\longrightarrow\mathbb{R}$$
 with the following properties:

 \begin{itemize}
 \item[\textit{i)}] There exists $M>0$ such that $|a(u,v)|\le M\|u\|_V\|v\|_V,\,\forall u,v\in V$;
 \item[\textit{ii)}] There exists $\mu,\alpha>0$ such that $a(u,u)+\mu\|u\|_H^2\ge \alpha\|u\|_V^2,\,\forall u\in V.$
 \end{itemize}
  One may associate  to the bilinear form $a$ an operator $\A:D(\A) =V\subset V^*\rightarrow V^*$ , defined by
  $$(\A u,v)_{V^*,V}=a(u,v), \, \forall u,v\in V. $$
  Consider the abstract parabolic Cauchy problem is
\begin{equation}\label{pbparab}
 \left\{
 \begin{array}{l}
 y'(t)+\A y(t)=g(t),\,t>0\\
 y(0)=y_0.
 \end{array}
 \right.
\end{equation}
 
 For $y_0\in H$ and   $g\in L^2(0, T; V^* )$ a variational solution is a function $y\in L^2(0,T; V)\cap   C([0,T];H)$  satisfying
 \begin{equation}
 \label{varsol}
 \langle y(t),v\rangle_H-\langle y_0,v\rangle_H+\int_0^ta(y(\tau),v)d\tau=\int_0^t\langle g(\tau),v\rangle_{V',V}d\tau,\forall t\in (0,T),\forall v\in V.
 \end{equation}
 
 Existence, uniqueness and continuous dependence of the variational solution is classic (see \cite{lionsmag}).
  The theory of $C_0$ semigroups of operators  inssures  that 
 $-\A$ generates an analytic semigroup $S(t)$  in $V^*$ and the variational solution coincides with the mild solution to nonhomogeneous problem \eqref{pbparab} which is given by Duhamel's formula:
 \begin{equation}\label{Amildsol1}
 	y(t)=S(t)y_0+\int_0^tS(t-s)g(s)ds=S(t)y_0+[S*g](t).
 \end{equation}

 

 
  Let also 
  \begin{equation}\label{Aflat}
  \A^\flat:D(\A^\flat)\subset H\rightarrow H,\, D(\A^\flat)=\{z\in V|\,\A z\in H\},\, \A^\flat z=\A z\mbox{ for } z\in D(\A^\flat).
  \end{equation}
 
Then $-\A^\flat$ is the generator of a   $C_0$, analytic semigroup in $H$, also denoted by $S(t)$ which is the restriction to $H$ of the semigroup generated in $V^*$.


 One may prove that if $y_0\in H$ and $g\in L^2(0,T;H)$ the mild solution given by \eqref{Amildsol1} coincides with the variational solution satisfying \eqref{varsol}.

 \subsubsection*{Irreducibility, analyticity and positivity improving properties}
 
 We consider now that $H=L^2(\Omega)$ for some open set $\Omega\subset\R^N$ .
 
 A $C_0$-semigoup $(S(t))_{t\geq 0}$ on $L^2(\Omega)$ is called \textit{irreducible} if given a Lebesgue  measurable subset $\omega\subset\Omega$, the set  $$L^2(\omega)=\{u\in L^2(\Omega)|u=0\,a.e.\,x\in\Omega\setminus\omega\}$$ 
 is invariant under the action of $(S(t))_{t\ge0}$ if and only if either $\mu(\omega)=0$ or $\mu(\Omega\setminus\omega)=0$.
  
  For a measurable function $f:\Omega\rightarrow\R$ one uses the following notations for various properties of  positivity:
  \begin{itemize}
  	\item[-] $f\geq 0$ if $f(x)\geq 0$  \textit{a.e.} $x\in\Omega;$
  	\item[-]$f> 0$ if $f\geq 0$ and $\mu(\lbrace x\in\Omega| f(x)\neq 0\rbrace)>0;$
  	\item[-] $f>> 0$ if $f\geq 0$ and $\mu(\lbrace x\in\Omega| f(x)= 0\rbrace)=0.$
  \end{itemize}
   
   One says that the semigroup $S(t)$ is positive if for $f\in H$, $f\ge0$ one has $S(t)f\ge0,\forall t\ge0$. Positivity of the semigroup may be characterized by the Beurling-Deny Condition (one denotes by $u^+=\max(u,0)$, $u^-=u^+-u$):
   \begin{theorem}
   \label{th-BD} The semigroup $(S(t))_{t\ge0}$ is positive if and only if
   \begin{equation}
 \label{condBD}u\in V \Rightarrow u^+\in V \text{ and } a(u^+,u^-)\le0,\,\forall u\in V. 
   \end{equation} 

    \end{theorem}
 Positivity improving properties of a positive semigroup occur under supplementary assumptions involving irreducibility and analyticity (Theorem 3 in \cite{majrob1981} or Theorem  10.1.2 in \cite{arendt_heat}):
 \begin{theorem}
 Assume $(S(t))_{t\geq 0}$ is positive, irreducible and holomorphic semigroup on $H=L^2(\Omega)$. Then, if $f\in H$ and $f>0$ one has
 \begin{equation}\label{th-pozimpr}
 S(t)f>>0 \quad \forall t>0.
 \end{equation}
 \end{theorem}

  \subsubsection*{Positivity for solutions to scalar parabolic problems}

  Consider $\Omega\subset\R^N$ an open, connected, bounded domain with smooth boundary and the parabolic initial boundary value problem
  \begin{equation}\label{pbparabmax}
     \begin{cases}
     D_ty+ {L}y=g  &\text{ in }Q\\
     \displaystyle\beta(x)\frac{\partial y}{\partial \nu_{A}}+\eta(x)y=0  &\text{ on }(0,T)\times\partial\Omega\\
     y(0,x)=y_0(x) &\text{ in }\Omega,
     \end{cases}
     \end{equation} where 
     
     \begin{itemize}
     \item[\textit{i)}] The elliptic part is  \begin{equation}\label{opparabdiv2}
        Lu=-\sum_{j,k=1}^N D_j(a^{jk}(x) D_ku)+\sum\limits_{k=1}^Nb^{k}(x) D_ku+c(t,x)u,
        \end{equation}
        with coefficient functions in the operator and in the boundary conditions  satisfying:
        \begin{itemize}
           \item  $a^{jk}\in W^{1,\infty}(\Omega) $  
            satisfy the ellipticity condition: $$\exists \mu>0 \text{ \textit{s.t.} }\sum_{j,k=1}^N a^{jk}(x)\xi_j\xi_k\geq\mu|\xi|^2,\quad\forall\xi\in\R^N,\quad(t,x)\in Q;$$
           \item $b=(b_1,\ldots,b_N)^\top, b^k\in L^\infty(\Omega)$  and $c\in L^\infty(Q)$;
           \item The coefficients in the boundary conditions satisfy $\beta\in C(\partial\Omega,\{0,1\})$, $\eta\in C^1(\partial\Omega)$, $\eta\geq 0$  and $\beta+\eta>0$ in $\partial\Omega$;
           
           \item Denoting by $A=A(x)=(a^{jk}(x))_{jk}$ the matrix valued functions constructed with the coefficients in the principal part, $
                      \frac{\partial}{\partial\nu_{A}}$  the corresponding conormal derivative.
           \end{itemize} 
           \item[\textit{ii)}] The source $g\in  L^2(0,T;L^2(\Omega))\simeq L^2(Q)$  and the initial data $y_0\in L^2(\Omega)$.
     \end{itemize}
  Denote by
  \begin{eqnarray}
  \Gamma_D=\{x\in\partial\Omega|\, \beta(x)=0\},\, \Gamma_{R}=\partial\Omega\setminus\Gamma_D,\nonumber\\
  H=L^2(\Omega),\, V=\lbrace v\in H^1(\Omega); v|_{\Gamma_D}=0\rbrace.\label{Vspace}
  \end{eqnarray}
       Consider the corresponding Gelfand triple
   $$V\subset H=L^2(\Omega)\subset V^*,$$
   and let    $$\tilde a:\R_+\times V\times V\longrightarrow\mathbb{R}, \tilde a(t,u,v)=a(u,v)+\int_\Omega c(t,x)uvdx$$
   with
   \begin{equation}\label{a-form}
   a(u,v)=\int_\Omega\langle A(x)\nabla u, \nabla v\rangle+\langle b, \nabla u\rangle v dx+\int_{\Gamma_R}\eta(x) u v d\sigma.
   \end{equation}   
   For $y_0\in H$ and   $g\in L^2(0, T; V^*)$ a variational solution is a function $y\in L^2(0,T; V)\cap   C([0,T];H)$  satisfying
   \begin{equation}
   \label{varsol1}
   \langle y(t),v\rangle_H-\langle y_0,v\rangle_H+\int_0^t\tilde a(\tau, y(\tau),v)d\tau=\int_0^t\langle g(\tau),v\rangle_{V^* ,V}d\tau,\forall t\in (0,T),\forall v\in V.
   \end{equation}
   
   \begin{remark}\label{rem_H2}
   Existence, uniqueness and continuous dependence on data of the variational solution for this non-autonomous problem is classical (see \cite{lionsmag}). 
   
   Denote, as before $\A:V\rightarrow V^*$ the operator associated to the form $a$ given in \eqref{a-form}. Consider also the associated operator $\A^\flat$ defined in \eqref{Aflat}. The regularity theory for general elliptic boundary value problems (see \cite{necas}, Theorems 1.1, 2.2) ensure that $D(\A^\flat)\subset H^2(\Omega)$.
   
   Also, $\A$ and $\A^\flat$ generate analytic semigroups in $V^*$ and respectively in $H$ and one may study existence and properties of solutions to nonautonomous  \eqref{pbparabmax} by treating it as a pertubed autonomous problem. 
   Positivity properties of the solutions for the nonautonomous case is considered in \cite{aredierouh2013} but,  as  we are interested in positivity improving properties, which are not a consequence of the results of the cited paper, in establishing the following result, which is a kind of strong maximum principle, we take this approach of analysing it as a perturbed autonomous system.    
   \end{remark}

   \begin{theorem}\label{strongmaxweaksol}
   	Assume that $y$ is variational solution to \eqref{pbparabmax} corresponding to positive initial data $y_0\ge0$ and positive source $g\ge0$. Then 
   	\begin{itemize}
   		\item $y(t,x)\ge0$ in $Q$;
   		\item If for some $\bar t>0$  the variational solution $y$ vanishes on a set of nonzero measure, \textit{i.e.}
   		\begin{equation}\label{vanish}
   			\mu\{x\in\Omega:y(\bar t,x)=0\}>0,
   		\end{equation}
   		then $y(t,x)=0$ and $g(t,x)=0$ for $0\le t<\bar t$ \textit{a.e.} $x\in\Omega$.
   	\end{itemize}
   \end{theorem} 
   \proof
 We write the problem  \eqref{pbparabmax}  in the abstract form
\begin{equation}
\label{eq-perturb}
\left\{
\begin{array}{l}
y'(t)+\A y(t)=g(t)+T(t)y(t),\,t>0\\
y(0)=y_0,
\end{array}
\right.
\end{equation}
where 
$$
T(t)\in L(H), [T(t)u](x)=-c(t,x)u(x),\, t>0, u\in H, \, x\in\Omega,
$$
and
$$\|T(t)\|_{L(H)}\le\|c\|_{L^\infty(Q)}.$$
With no loss of generality we may assume that $c(t,x)\le 0$ in $Q$ this beeing achieved by stydying the problem satisfied by $z=ye^{\gamma t} $ with $\gamma>2\|c\|_{L^\infty(Q)}$. With this assumption $T(t)$ is a positive operator for all $t>0$ \textit{i.e.} if $f\in H, f\ge0$ we have $T(t)f\ge0,\forall t>0$.

  Denote by $S(t)$ the restriction to $H$ of the  semigroup in $V^*$ generated by $-\A$.  
  
  Observe that if $u\in V$ then the positive part $u^+=\max\{u,0\}\in V$ and 
   $$
   a(u^+,u^-)=0.
   $$
   Consequently, by the Beurling--Deny criterion Theorem \ref{th-BD}, one concludes that $S(t),t>0$ is positive.
  
   Moreover, as $C_0^\infty(\Omega)\subset V\subset H^1(\Omega)$ one knows that if $\omega\subset\Omega$ is a measurable set then
   $$
   1_\omega V\subset V
   $$
   implies that $\mu(\omega)=0$ or $\mu(\Omega\setminus\omega)=0$ (see \cite{arendt_heat}).
   As a consequence $S(t) $ is   irreducible (Theorem 10.1.5 in \cite{arendt_heat}).
   Now, as $S(t)$ is analytic semigroup in $H$ one has that $S(t)$ is positivity improving, by Theorem \ref{th-pozimpr}.

   We study now existence for the Cauchy problem \eqref{eq-perturb} by a fixed point argument. Let $\delta>0$ to be fixed later. For $z\in L^2(0,\delta;H)$ define $y\in L^2(0,\delta;H)$ to be the solution to problem
   \begin{equation}
   	\label{eq-perturb1}
   	\left\{
   	\begin{array}{l}
   		y'(t)+\A y(t)=g(t)+T(t)z(t),\,t>0\\
   		y(0)=y_0,
   	\end{array}
   	\right.
   \end{equation}
   The variational solution to \eqref{eq-perturb1}, which is also a mild solution, has the form
   \begin{equation}
   	\label{Tpfix}
   	y(t)=S(t)y_0+\int_0^tS(t-s)[g(s)+T(s)z(s)]ds=:[\mathcal{T}z](t)
   \end{equation}
   A solution to \eqref{eq-perturb} on $[0,\delta]$ is a fixed point of operator $\mathcal{T}$ in $L^2(0,\delta;H)$. It is a standard argument to show that, for $\delta>0$ small enough, depending only on $\|T(t)\|_{L^\infty(0,\delta;L(H))},$ $\mathcal{T}$ is a contraction in $L^2(0,\delta;H)$. Indeed, knowing that the semigroup $S$ has an at most exponential growth, \textit{i.e.} there exists $M,\gamma>0$ such that
   $$
   \|S(t)\|_{L(H)}\le Me^{\gamma t},\,t>0,
   $$
   we have
   $$
   \|\mathcal{T}(z_1-z_2)\|_{L^2(0,\delta;H)}=\|S*[T(\cdot)(z_1(\cdot)-z_2(\cdot))]\|_{L^2(0,\delta;H)}\le
   $$
   $$
   \le\|S(\cdot)\|_{L^1(0,\delta;L(H))}\cdot\|z_1-z_2\|_{L^2(0,\delta;H)}\le \frac M\gamma (e^{\gamma \delta}-1)\cdot\|z_1-z_2\|_{L^2(0,\delta;H)}
   $$
   and $\mathcal{T}$ is a contraction is $\delta $ is chosen such that $\frac M\gamma (e^{\gamma \delta}-1)<1$ and Banach fixed point theorem implies existence and uniqueness of the solution to \eqref{eq-perturb} on $[0,\delta]$. The solution is then extended to $\R_+$  by solving Cauchy problems on successive intervals of length $\delta$.
   
   Observe now that by the positivity of $T(t)$ and of the semigroup $S(t)$ we have that, for positive source $g\ge0$ and positive initial data $y_0\ge0$, if $z\in L^2(0,\delta;H_+)$ with $H_+=\{u\in H|u\ge0\}$,
   \begin{equation}
   	\label{Tpfix1}
   	[\mathcal{T}z](t)=S(t)y_0+\int_0^tS(t-s)[g(s)+T(s)z(s)]ds\ge0
   \end{equation}   
   as  sum of two positive terms; thus $L^2(0,\delta;H_+)$ is an invariant closed set under the action of $\mathcal{T}$. Consequently, the fixed point belongs to this set and thus the solution $y$ to \eqref{eq-perturb} is positive.

   Assume now that the   solution $y$ corresponding to positive initial data $y_0$ and positive source $g$, vanishes at some moment   $\bar t>0$ on a set of nonzero measure, \textit{i.e.} 
   \eqref{vanish} holds, and observe that choosing as initial time some $0<t_0<\bar t$ we have 
   $$
   y(\bar t)=S(\bar t-t_0)y(t_0)+\int_{t_0}^{\bar t}S(\bar t-s)[g(s)+T(s)z(s)]ds.
   $$
  So, by the positivity of the semigroup, the second term above is positive and thus   $y(\bar t,\cdot)\ge S(\bar t-t_0)y(t_0)\ge0$, which implies
  $$
  \{x\in\Omega:y(\bar t,x)=0\}\subset \{x\in\Omega:[S(\bar t-t_0)y(t_0)](x)=0\}.
  $$
Thus, by \eqref{vanish}, $\mu(\{x\in\Omega:[S(\bar t-t_0)y(t_0)](x)=0\})>0$ and by positivity improving property of the semigroup   necessarily we must have  $y(t_0, \cdot)=0$ \textit{a.e.} in $\Omega$.
   As $t_0$ is arbitrarily chosen in $(0,\bar t)$, we conclude that $y(t,x)=0$ and consequently also  $g(t,x)=0$ for $0\le t<\bar t$ \textit{a.e.} $x\in\Omega$.\fin
   

   

 \subsubsection*{Weakly coupled  linear parabolic systems. }

 Consider the linear parabolic system \eqref{sysinitial} with the respective assumptions concerning the coefficients of the operators. For sources $g_i\in L^2(0,T;\Omega)$ we describe the variational setting of the problem. Denote by
    $$\Gamma_{i,D}=\{x\in\partial\Omega|\, \beta_i(x)=0\},\, \Gamma_{i,R}=\partial\Omega\setminus\Gamma_{i,D}.$$ 
      The variational formulation for the problem needs the Hilbert spaces:  
      $$ H=L^2(\Omega),\, V_i=\lbrace v\in H^1(\Omega); v|_{\Gamma_{i,D}}=0\rbrace,$$
      $$
      \textbf{H}=H\times\cdots_n\times H,\textbf{V}=V_1\times\cdots_n\times V_n
      $$
     with the canonical Hilbert structures of  product spaces, and the corresponding Gelfand triple
    $$\textbf{V}\subset \textbf{H}\subset \textbf{V}^*.$$
    The bilinear form $\textbf{a}:\textbf{V}\times\textbf{V}\rightarrow \R$ is defined, for $\textbf{u}=(u_1,\ldots,u_n), \textbf{v}=(v_1,\ldots,v_n)\in  \textbf{V}$ as
   \begin{equation}\begin{array}{l}\label{avect}
   \textbf{a}(\textbf{u},\textbf{v})=a_1(u_1,v_1)+ \ldots +a_n(u_n,v_n),\quad  a_i: V_i\times V_i\longrightarrow\mathbb{R},\\
   \\
\displaystyle a_i(u_i,v_i)=\int_\Omega\langle A_i(x)\nabla u_i, \nabla v_i\rangle+\langle \textbf{b}_i,\nabla u_i\rangle v_idx+\int_{\Gamma_{R_i}}\eta(x) u_i v_i d\sigma,\ u_i,v_i\in V_i.
   \end{array}
     \end{equation}    
    For $\textbf{y}_0=(y_{1,0},\ldots,y_{n,0})\in \textbf{H}$ and   $\textbf{g}=(g_1,\ldots,g_n)\in L^2(0, T; \textbf{V}^*)$ a variational solution to \eqref{sysinitial} is a function $\textbf{y}\in L^2(0,T; \textbf{V})\cap   C([0,T];\textbf{H})$  satisfying
    \begin{equation}
    \label{varsolsyst}\begin{array}{l}

    \displaystyle\langle \textbf{y}(t),\textbf{v}\rangle_\textbf{H}-\langle \textbf{y}_0,\textbf{v}\rangle_H+\int_0^t \textbf{a}(\textbf{y}(\tau),\textbf{v})d\tau+\int_0^t\sum_{i=1}^n\langle L_i^0(\tau)\textbf{y}(\tau),v\rangle_Hd\tau=\\
    \\
    \displaystyle=\int_0^t\langle \textbf{g}(\tau),\textbf{v}\rangle_{\textbf{V}^* ,\textbf{V}}d\tau,\quad\,\forall t\in (0,T),\forall \textbf{v}\in V.
    \end{array}
    \end{equation}
    
    \begin{remark}\label{rem_H22}
       Existence, uniqueness and continuous dependence on data of the variational solution $\textbf{y}\in C([0,T]; \textbf{H})\cap L^2(0,T;\textbf{V})$ for this non-autonomous problem is also a consequence of the results in  \cite{lionsmag}. 
       
       Denote, as before $\mathbf{\A}:\textbf{V}\rightarrow \textbf{V}^*$ the operator associated to the form $\textbf{a}$ given in \eqref{a-form}. Consider also the associated operator $\A^\flat$ defined in \eqref{Aflat}. Again, the regularity theory for general elliptic boundary value problems (see \cite{necas}) ensure that $D(\mathbf{\A}^\flat)\subset (H^2(\Omega))^n$. Denote by $\textbf{S}(t)$ the analytic semigroup generated by $\mathbf{\A}$ or $\mathbf{\A}^\flat$ in $V$ respectively in $\textbf{H}$ and analyse the system \eqref{sysinitial} as a perturbed autonomous equation in the next result which is a (strong) invariance principle of the cone of $L^2$ positive functions under the action of the flow associated to a class of weakly coupled parabolic systems.

       \end{remark}

       \begin{theorem}\label{strmaxsys}
       	Assume that $\textbf{y}\in C([0,T]; \textbf{H})\cap L^2(0,T;\textbf{V})$ is variational solution to \eqref{sysinitial} corresponding to positive initial data $\textbf{y}_0=(y_{i,0})_i\ge0$ and positive source $\textbf{g}=(g_{i})_i\in L^2(Q), \textbf{g}\ge0$. Assume
       	also that
       	$$c_{il}\leq 0,  \quad i,l=\overline{1,n}, i\neq l.$$       	
       	 Then 
       	\begin{itemize}
       		\item[i)] $\textbf{y}(t,\cdot)\ge0$ for $t\in[0,T]$;
       		\item[ii)] If for some $\bar t>0$  the variational solution $\textbf{y}$ vanishes on a set of nonzero measure, \textit{i.e.}
       		\begin{equation}\label{vanish1}
       			\mu\{x\in\Omega:\textbf{y}(\bar t,x)=0\}>0,
       		\end{equation}
       		then $\textbf{y}(t,x)=0$ and $\textbf{g}(t,x)=0$ for $0\le t<\bar t$.
       	\end{itemize}
       \end{theorem} 
       \begin{proof}
     We write the problem  \eqref{sysinitial}  as
     \begin{equation}\label{sysinitial1}
             \left\lbrace
             \begin{array}{ll}
             D_ty_i-\Delta_{A_i}y_i+L^1_iy_i=g_i-L_i^0(t)\textbf{y}, &(0,T)\times\Omega,
             \\
             \beta_i(x)\frac{\partial y_i}{\partial \nu_{A_i}}+\eta_i(x)y_i=0 ,  &(0,T)\times\partial\Omega, 
             \end{array},\quad i=\overline{1,n}.
             \right.\end{equation} 
     and in       the abstract form
    \begin{equation}
    \label{sys-perturb}
    \left\{
    \begin{array}{l}
    \textbf{y}'(t)+\mathbf{\A} \textbf{y}(t)=g(t)+\textbf{T}(t)\textbf{y}(t),\,t>0\\
   \textbf{ y}(0)=\textbf{y}_0,
    \end{array}
    \right.
    \end{equation}
    where 
    $$
   \textbf{ T}(t)\in L(\textbf{H}), [\textbf{T}(t)\textbf{u}]=(-L_i^0(t)\textbf{u})_i,\, t>0, \textbf{u}\in \textbf{H},
    $$
    where the operators $L_i^0(t)\in L(\textbf{H}, H)$ are given 
    $$
   [ L^0_i(t)\textbf{u}](x)=\sum\limits_{l=\overline{1,n}}c_{il}(t,x) u_l(x),\,x\in\Omega.
    $$
    Observe that    
        $$\|\textbf{T}(t)\|_{L(H)}\le\|\textbf{c}\|_{L^\infty(Q)}=\max_{i,l}\|c_{il}\|_{L^\infty(Q)}.$$
    With no loss of generality we may assume that $\forall i, c_{ii}(t,x)\le 0$ in $Q$ as we may equivalently study positivity   $z=ye^{\gamma t} $ with $\gamma>2\|\textbf{c}\|_{L^\infty(Q)}$. 
    
    Existence for the Cauchy problem \eqref{sys-perturb} may also be treated, as in the scalar case,  by a fixed point argument.      
    For $\textbf{z}\in L^2(0,\delta;\textbf{H})$ define, as in the scalar case,  $\textbf{y}\in L^2(0,\delta;H)$ to be the solution to problem
           \begin{equation}
           	\label{sys-perturb1}
           	\left\{
           	\begin{array}{l}
           		\textbf{y}'(t)+\mathbf{\A}\textbf{y}(t)=\textbf{g}(t)+\textbf{T}(t)\textbf{z}(t),\,t>0\\
           		\textbf{y}(0)=\textbf{y}_0.
           	\end{array}
           	\right.
           \end{equation}
          Using the formulation as a mild solution,  the variational solution  to \eqref{sys-perturb1} appears to be a fixed point of the operator 
           \begin{equation}
           	\label{Tsyspfix}
           [\mathbf{\mathcal{T}}\textbf{z}](t)=\textbf{S}(t)\textbf{y}_0+\int_0^t\textbf{S}(t-s)[\textbf{g}(s)+\textbf{T}(s)\textbf{z}(s)]ds
           \end{equation}
As in the  scalar case one finds a $\delta>0$ such that  $\mathbf{\mathcal{T}}$ is contraction in $L^2(0,\delta;\textbf{H})$ and then solve the successive Cauchy problems on intervals pf length $\delta$, to cover all $[0,T]$.

Observe now that, as $\forall i,l,\,c_{il}\le0$, if $\textbf{y}_0\ge0$,  $\textbf{g}(s)\ge0$ and $z(s)\ge0,\, a.e.\,s\in[0,T]$ we have that $  [\mathbf{\mathcal{T}}\textbf{z}](t)\ge0,\, a.e.\,t\in[0,T]$ and thus, the closed cone 
$L^2(0,\delta;\textbf{H}_+)$ is invariant under the action of the operator $\mathbf{\mathcal{T}}$. Consequently, its' fixed point, which is solution to \eqref{sysinitial},\eqref{sys-perturb}, must be positive.

Positivity improving property may be now obtained by analysing each equation independently, exactly as in the scalar case. \end{proof}

We may now formulate a positivity improving result concerning semilinear systems.
 \begin{theorem}\label{strmaxsys1}
	Consider system \eqref{rd-sys} with the operator and boundary conditions satifying hypothesis (H2). Assume that the nonlinearities $\textbf{f}=(f_i)_i$, $f_i:\R\times\overline\Omega\times\R^n\longrightarrow\R$ are $C^1$ smooth, and 

\begin{equation}\label{hypinv}
f_i(t,x,y_1,\ldots,y_{i-1},0,y_{i+1},\dots, y_n)\leq 0,  i =\overline{1,n},   t>0,x\in\overline\Omega,\,\textbf{ y}\ge0.
\end{equation}
		Assume that $\textbf{y}\in C([0,T]; \textbf{H})\cap L^2(0,T;\textbf{V})\cap L^\infty(Q)$ is variational solution  corresponding to positive initial data $\textbf{y}_0=(y_{i,0})_i\ge0$ and positive source $\textbf{g}=(g_{i})_i\in [L^2(Q)]^n, \textbf{g}\ge0$.
	Then 
	\begin{itemize}
		\item[i)] $\textbf{y}(t,\cdot)\ge0$ for $t\in[0,T]$;
		\item[ii)] If for some $\bar t>0$  the variational solution $\textbf{y}$ vanishes on a set of nonzero measure, \textit{i.e.}
		\begin{equation}\label{vanish2}
			\mu\{x\in\Omega:\textbf{y}(\bar t,x)=0\}>0,
		\end{equation}
		then $\textbf{y}(t,x)=0$ and $\textbf{g}(t,x)=0$ for $0\le t<\bar t$.
	\end{itemize}
\end{theorem} 
   \begin{proof}
    
 Observe that, by a linearization mechanism,  $\textbf{y}$ is variational solution to the "linearized" system
    \begin{equation}\label{rd-sys-lin-inv}
    	\left\lbrace
    	\begin{array}{ll}
    		D_ty_i-\sum\limits_{j,k=1}^N D_j(a_i^{jk} D_ky_i)+\sum\limits_{k=1}^N b_i^k D_ky_i+c^{\textbf{y}}_i(t,x)y_i=g_i+\bar g_i  &\text{in }(0,T)\times\Omega,
    		\\
    		\displaystyle\beta_i(x)\frac{\partial y_i}{\partial \nu_{A_i}}+\eta_i(x)y_i=0  &(0,T)\times\partial\Omega,  \\
    		
    	\end{array} i\in\overline{1,n},
    	\right.\end{equation} 
    where 
    $$
    c^{\textbf{y}}_i(t,x)=\int_0^1\frac{\partial}{\partial y_i}f_i(t,x,y_1(t,x),\ldots,y_{i-1}(t,x),\tau y_i(t,x),y_{i+1}(t,x),\dots, y_n(t,x))d\tau,
    $$
    and 
    $$
    \bar g_i(t,x)= -f_i(t,x,y_1(t,x),\ldots,y_{i-1}(t,x),0,y_{i+1}(t,x),\dots, y_n(t,x))d\tau.
    $$
    As $\textbf{y}\in [L^\infty(Q)]^n$ we have that $ c^{\textbf{y}}_i,\bar g_i\in L^\infty(Q),\forall i$ and by hypothesis \eqref{hypinv} we have that $\bar g_i\ge 0$. The conclusion now follows by Theorem \ref{strmaxsys}.\end{proof}
    
     \begin{comments}
     
     We based our presentation of semigroups defined by forms and corresponding positivity and invariance properties  on the lecture notes of Wolfgang Arendt, Heat Kernels \cite{arendt_heat}. 
          We refer to \cite{arendt_heat}, \cite{ouh2005} for results concerning positivity of semigroups generated by forms and characterization of invariance of closed, convex sets under the action of the flow.
     
     Positivity improving properties of the semigroups were studied in the papers of O. Bratteli, A. Kishimoto, D. Robinson \cite{kisrob1} and A. Kishimoto, D. Robinson \cite{kisrob2}; the characterisation of positivity improving property using analyticity of the semigroup combined with irreducibility was established by A. Majewski, D. Robinson \cite{majrob1981} in the more general framework of semigroups in Banach lattices.
     
     The invariance properties and the strong maximum type principles  established here in the variational setting given by forms are versions of the results in a classical setting from \cite{wein},\cite{amann}, \cite{martin} with the mention that our approach is not suitable for time dependent coefficients in the principal part.
    
     We also mention the paper \cite{aredierouh2013} where the invariance of convex sets under flows generated by non-autonomous flows is studied but, while this paper may be used for the positivity of solutions in our situation,  positivity improving results we establish are not consequences of this work.

      \end{comments}
      

 \section{$L^2$ Carleman estimates with boundary observations}
 \label{sec:Carleman}
The  Carleman estimates we establish for parabolic problems with homogeneous boundary conditions and boundary observations are new  in the case of Neumann and Robin boundary conditions and, consequently  we treat here the case of the Robin boundary conditions which covers both situations.
The principal ideas of the proof are very similar to the classic situation of parabolic problems with homogeneous boundry conditions and internal observations  (we refer here to \cite{furima1996}) but due to the particular geometry of the domain and the corresponding choice of auxiliary functions we have to focus on the particular treatment of boundary integrals.

 So, we start with  the linear parabolic system with homogeneous boundary conditions in a domain $\Omega$ with geometry described in hypothesis (H1):
    \begin{equation}\label{Aecinitial2}
   \left\lbrace
   \begin{array}{ll}
   D_ty-\Delta_Ay=\overline{g},&\, \text{ in }(0,T)\times\Omega,\\
   &\\
   \dfrac{\partial y}{\partial \nu_{A}}+\eta(x)y=0 , &\, \text{ on }(0,T)\times\partial\Omega,  \\
   &\\
   y(0,\cdot)=y^0, &\, \text{ in }\Omega,\\
   \end{array}
   \right.\end{equation} 
   with 
   \begin{itemize}
   
   \item $A=(a_{ij})_{ij}, a_{ij}\in W^{1,\infty}(\Omega), A=A^\top$ and the operator $\DeltaA$ is given by
   $$[\DeltaA{u}](x)=\text{div}(A(x)\nabla u(x)), x\in\Omega;$$
   \item $\left[\dnuA u\right](x)=\sprod{A(x)\nabla u(x)}{\nu(x)},\, x\in\partial\Omega;$
   \item $\eta\in L^\infty(\partial\Omega), \eta\ge0,$
   \item $\bar g\in L^2(0,T;L^2(\Omega)).$
   \end{itemize}
   
  Now, for the  simplicity of formulations and of computations we introduce the following notations: 
  
  \noindent For $u,v:\Omega\ra\R, F:\Omega\ra\R^n$ with appropriate regularity, denote by 
    \begin{itemize}
     \item $\sprodA{\nabla u}{\nabla v}(x)=\langle A(x)\nabla u(x), \nabla v(x)\rangle,\, x\in\Omega$;
    \item  $\modA{\nabla u}(x)=\sqrt{\langle A(x)\nabla u(x), \nabla v(x)\rangle}$;
    \item  $[\dvA{F}](x)=\text{div }(A(x)F(x)),\, x\in\Omega$.
    
    \end{itemize}
   
   Direct computations show that:
   \begin{itemize}
   \item $\dvA(u\nabla v)=u\DeltaA{v}+\sprodA{\nabla u}{\nabla v}$;
   \item  $\dvA(uF)=\text{div}(AuF)=u\,\dvA F+\sprodA{F}{\nabla u}$;
   \item $\DeltaA({uv})=\text{div }(A\nabla(uv))=\text{div }(vA\nabla u+uA\nabla v)=u\DeltaA{v}+v\DeltaA{u}+2\sprodA{\nabla u}{\nabla v}$.
   \end{itemize}
   
   One easily verifies the Gauss-Ostrogradski-Green type formulas: for $F\in H^1(\Omega;\R^N),u\in H^1(\Omega),v\in H^2(\Omega)$ we have
   \begin{eqnarray}
   \int_{\Omega}\text{div}_A Fdx=\int_{\Omega}\text{div}(AF)dx=\int_{\partial\Omega}\langle AF, \nu\rangle d\sigma=\int_{\partial\Omega}\sprodA{F}{\nu}d\sigma\\
   \int_{\Omega}u\DeltaA{v}dx=-\int_\Omega\sprodA{\nabla u}{\nabla v}+ \dvA(u\nabla v)dx=
-\int_\Omega\sprodA{\nabla u}{\nabla v}dx+ \int_{\partial\Omega} u\dnuA{v}d\sigma.
   \end{eqnarray}
   
  Concerning the auxiliary function $\psi_0$ given in \rm{(H1)} ( existence of which is related to the geometry of the domain $\Omega$) we make the following remark: 
  \begin{remark}\label{remsubharmpsi}
  One may assume that $\psi_0$ satisfies
 \begin{equation}\label{pozdelta}
 \Delta_A \psi_0>0 \text{ in }\overline \Omega.
 \end{equation}
  Indeed, if $\psi_0$ with properties \eqref{psi0} exists, we may replace $\psi_0$ with $e^{\mu\psi_0}$, enjoying also \eqref{psi0}, and,  moreover, for $\mu>\mu_0>0$ big enough
  $$
  \Delta_A e^{\mu\psi_0}=(\mu^2|\nabla \psi_0|_A^2+\mu\Delta_A \psi_0)>0.
  $$ 
  
  \end{remark}  
Observe that if we take in Remark     the function $\psi_0$ given in \rm{(H1)}

Introduce now, similarly to the standard case the following auxiliary functions constructed by using the function $\psi_0$ given in \rm{(H1)}.   
Let
  $\psi=\psi_0+K$, with a constant $K>0$ big enough such that 
 $$\frac{\sup\psi}{\inf\psi}\leq\frac{8}{7}.$$  
 Consider now the auxiliary functions
  $$\varphi(t,x):=\frac{e^{\lambda\psi(x)}}{t(T-t)},\quad
 \alpha(t,x):=\frac{e^{\lambda\psi(x)}-e^{1.5\lambda\|\psi\|_{C(\overline{\Omega})}}}{t(T-t)}.$$
   This choice for the function $\psi$ is appropriate in order to have the following estimates for time derivatives of auxiliary functions:
  	$$|\varphi_t|\leq C\varphi^2,\quad |\alpha_t| \leq C\varphi^2,\quad |\alpha_{tt}| \leq C\varphi^3.$$
 uniformly in $ \lambda>0$, with a constant  $C=C(T)$.
 
 Here we denoted by $\alpha_t=D_t\alpha$ and in computations we will use variables as indices in order  to indicate corresponding partial derivatives.

  Let also 
   \begin{equation}
   \tilde{\psi}=2K-\psi,
   \end{equation} 
   and the corresponding weight functions
   \begin{equation}\label{fialfatilda}
   \tilde{\varphi}(t,x):=\frac{e^{\lambda\tilde{\psi}(x)}}{t(T-t)},\quad
   \tilde{\alpha}(t,x):=\frac{e^{\lambda\tilde{\psi}(x)}-e^{1.5\lambda\|\psi\|_{C(\overline{\Omega})}}}{t(T-t)}.
   \end{equation}
  for which we also have
    $$|\tilde{\varphi}_t|\leq C\tilde{\varphi}^2,\quad|\tilde{\alpha}_t| \leq C\tilde{\varphi}^2,\quad |\tilde{\alpha}_{tt}| \leq C\tilde{\varphi}^3.$$ 
 
 \medskip
 
\noindent  We describe now the observation operator we will consider. Let 
 $$ \gamma,\delta\in L^\infty(\Gamma_1)$$
 for which there exists some positive constant $\epsilon>0$ such that 
  $$
  \gamma(x)\eta(x)-\delta(x)\ge\epsilon>0,\quad x\in\Gamma_1.
  $$
 The observation operator is defined by 
 $$\zeta:H^2(\Omega)\rightarrow L^2(\Gamma_1),\, \zeta(y)=\gamma \frac{\partial }{\partial\nu_A}y+\delta y.$$
 The conditions imposed to $\gamma,\delta$ express the fact that  the observation operator and the boundary operator have null intersection of corresponding kernels and there exists $K>0$ such that for the solutions to \eqref{Aecinitial2} one has 
 \begin{equation}
 \label{zetagamma1}
 |y(t,x)|+\left|\dnuA{y}(t,x)\right|\le K |\zeta(y)(t,x)|,\quad x\in\Gamma_1,t>0. 
 \end{equation}
 \medskip
 
 \noindent The main result concerning Carleman estimates is:
 \begin{proposition}\label{lemaCarlemanclassic} For $g\in L^2(Q)$,
 there exist constants $\lambda_0=\lambda_0(\Omega),$ $s_0=s_0(\Omega)$ such that, for any $\lambda\geq\lambda_0$, $ s\geq s_0$ and some $C=C(T,\Omega)$, the following inequality holds:
 \begin{equation}\label{classicalCarleman}
 \begin{aligned}
 &\int_{Q}\left[s\lambda^2\varphi|\nabla y|^2+s^3\lambda^4\varphi^{3}|y|^2\right]e^{2s\alpha}dxdt+
 \int\limits_{\Sigma_0}  s^2\lambda^2\varphi^2 |y|^2e^{2s\alpha}d\sigma dt
 \\
 &\leq
 C\left(\int_{Q}|\bar g|^2e^{2s\alpha}dxdt+
 \int\limits_{\Sigma_1}s^3\lambda^3\varphi^{3}|\zeta(y)|^2e^{2s\alpha}d\sigma dt\right),\qquad\quad
 \end{aligned}
 \end{equation}
 for $y\in [H^1(0,T; L^2(\Omega))\cap L^2(0,T; H^2(\Omega))]^n$ solutions of \eqref{sysinitial}.
 \end{proposition}

  \begin{proof}

   	We denote by $w:=ye^{s\alpha}$   and we consider the corresponding parabolic  problem satisfied by $w$: 
\begin{equation}\label{pbw}\left\lbrace
  	\begin{aligned}
  	&w_t-\Delta_Aw+2s\langle\nabla\alpha,\nabla w\rangle_A-(s^2|\nabla\alpha|_A^2-s\Delta_A\alpha+s\alpha_t)w=\g e^{s\alpha} &\, \text{ in }(0,T)\times\Omega\\
  	&\frac{\partial w}{\partial \nu_A}+\left(\eta(x)-s\lambda\varphi\frac{\partial\psi}{\partial \nu_A}\right)w=0  &\, \text{ on }(0,T)\times\partial\Omega
  	\end{aligned}\right.
  	\end{equation}   	
   	If we denote by 
   	\begin{equation*}
  	\begin{aligned}
  	X_s=X_s(t,x)w:=&\Delta_Aw+(s^2|\nabla\alpha|_A^2+s\Delta_A\alpha+s\alpha_t)w\\
  	X_a=X_a(t,x)w:=&-w_t-2s(\langle\nabla\alpha,\nabla w\rangle_A+s\Delta_A\alpha w)
  	\end{aligned}
  	\end{equation*}
  	we have
  	\begin{equation}\label{pbwX}\left\lbrace
  	\begin{aligned}
  	&X_s+X_a=-\g e^{s\alpha} &\, \text{ in }(0,T)\times\Omega\\
  	&\frac{\partial w}{\partial \nu_A}+\left(\eta(x)-s\lambda\varphi\frac{\partial\psi}{\partial \nu_A}\right)w=0  &\, \text{ on }(0,T)\times\partial\Omega
  	\end{aligned}\right.
  	\end{equation}
  	We multiply scalarly in $L^2(Q)$ equation \eqref{pbwX} by $X_a$ 
  	to obtain
  	
  	\begin{equation}\label{estfi}
  \int_QX_sX_adxdt\leq\int_Q(X_a+X_s)^2dxdt=\int_Q\overline{g}^2e^{2s\alpha}dxdt
  	\end{equation}
  	with
  	\begin{equation}\label{multiplL2}
  	\begin{aligned}
  	\int_QX_sX_adxdt=&\\
  	&-\int_Qw_t\Delta_Awdxdt-2s\int_Q\Delta_A w\langle\nabla\alpha,\nabla w\rangle_Adxdt-2s\int_Q\Delta_A w(\Delta_A\alpha) wdxdt\\
  	&-\int_Q(s^2|\nabla\alpha|_A^2+s\Delta_A\alpha+s\alpha_t)ww_tdxdt\\
  	&-2s\int_Q\langle\nabla\alpha,\nabla w\rangle_A(s^2|\nabla\alpha|_A^2+s\Delta_A\alpha+s\alpha_t)wdxdt\\
  	&-2s\int_Q(\Delta_A\alpha)w(s^2|\nabla\alpha|_A^2+s\Delta_A\alpha+s\alpha_t)wdxdt\\
  	&=T_1(Q)+T_2(Q)+T_3(Q)+T_4(Q)+T_5(Q)+T_6(Q).
  	\end{aligned}
  	\end{equation}
  	
  	We proceed as usually when treating Carleman estimates: we develop each of the terms above, using Green formulas, and we  emphasize the  dominant terms in parameters $s,\lambda$ and powers of $\varphi$ in the coefficients of $|w|^2$ and $\modA{\nabla w}^2$. More precisely, we will  find the dominant term in the integrals on $Q$ containing $|w|^2$ as $s^3\lambda^4\varphi^3|\nabla\psi|^4|w|^2$ and denote by $l.o.t(s^3\lambda^4\varphi^3 |w|^2)$ a term which is bounded in $Q$ by   $C(s^2\lambda^4\varphi^2+s^3\lambda^3\varphi^3 )|w|^2$ for some constant $C>0$ and $s,\lambda>0$ big enough. 
  	The dominant term in $|\nabla w|^2$ will be found to be $s\lambda^2\varphi|\nabla\psi|^2|\nabla w|^2$ and $l.o.t.(s\lambda^2\varphi|\nabla w|^2)$ denotes a term which is bounded in $Q$ by   $C(s\lambda\varphi+\lambda^2 )|\nabla w|^2$ for some constant $C>0$ and $s,\lambda>0$ big enough.

  	\begin{equation}\label{t1}
  	\begin{aligned}
  	T_1(Q)&=-\int_Qw_t\Delta_Awdxdt=\int_Q\langle\nabla w, (\nabla w)_t\rangle_Adxdt-\int_\Sigma\frac{\partial w}{\partial \nu_A}w_td\sigma dt\\
  	&=\frac{1}{2}\int_Q(|\nabla w|^2_A)_tdxdt-\int_\Sigma\frac{\partial w}{\partial \nu_A}w_td\sigma dt =-\int_\Sigma\frac{\partial w}{\partial \nu_A}w_td\sigma dt .\\
  	\end{aligned}
  	\end{equation}
  	
  	\begin{equation}\label{t2}
  	\begin{aligned}
  	T_2(Q)&=-2s\int_Q\Delta_A w\langle\nabla\alpha,\nabla w\rangle_Adxdt\\
  	&=2s\int_Q\langle\nabla w, \nabla\langle\nabla\alpha,\nabla w\rangle_A\rangle_Adxdt-2s\int_\Sigma\frac{\partial w}{\partial \nu_A}\langle\nabla\alpha,\nabla w\rangle_A d\sigma dt\\
  &	=  \int_Qs\sprodA{\nabla \modA{\nabla w}^2}{\nabla\alpha}dxdt +2\int_Q s\lambda^2\varphi|\langle\nabla w, \nabla\psi\rangle_A|^2dxdt        \\
&+\int_Q\lotnw dxdt-2s\int_\Sigma\frac{\partial w}{\partial \nu_A}\langle\nabla\alpha,\nabla w\rangle_A d\sigma dt\\
  	&=-\int_Q s\lambda^2\varphi|\nabla\psi|_A^2|\nabla w|_A^2dxdt+2\int_Q s\lambda^2\varphi|\langle\nabla w, \nabla\psi\rangle_A|^2dxdt\\
  	&+\int_Q\lotnw dxdt+s\int_\Sigma|\nabla w|_A^2\frac{\partial\alpha}{\partial \nu_A}d\sigma dt-2s\int_\Sigma\frac{\partial w}{\partial \nu_A}\langle\nabla\alpha,\nabla w\rangle_Ad\sigma dt.
  	\end{aligned}
  	\end{equation}
  	
  	\begin{equation}\label{t3}
  	\begin{aligned}
  	T_3(Q)&=-2s\int_Q\Delta_A w(\Delta_A\alpha) wdxdt\\
  	&=2s\int_Q\langle\nabla w,\nabla(\Delta_A\alpha w)\rangle_Adxdt-2s\int_\Sigma\frac{\partial w}{\partial \nu_A}\Delta_A\alpha wd\sigma dt\\
  	&=2s\int_Q\Delta_A\alpha|\nabla w|_A^2dxdt+2s\int_Q\langle\nabla w,\nabla(\Delta_A\alpha) w\rangle_Adxdt-2s\int_\Sigma\frac{\partial w}{\partial \nu_A}\Delta_A\alpha wd\sigma dt\\
  	&=2s\lambda^2 \int_Q\varphi\modA{\nabla\psi}^2|\nabla w|_A^2 dxdt+\int_Q\lotwnw  dxdt\\
&  	  -2s\int_\Sigma\frac{\partial w}{\partial \nu_A}\Delta_A\alpha wd\sigma dt,
  	\end{aligned}
  	\end{equation}
  where we have used the estimate
  $$\begin{aligned}
 & |s\langle\nabla w,\nabla(\Delta_A\alpha) w\rangle_A|\le Cs\lambda^3\varphi |w||\nabla w|\\
& \le C(s^2\lambda^4\varphi^2|w|^2+\lambda^2|\nabla w|^2)=\lotwnw.
  \end{aligned}
    $$

  	\begin{equation}\label{t4}
  	\begin{aligned}
  	T_4(Q)&=-\int_Q(s^2|\nabla\alpha|_A^2+s\Delta_A\alpha+s\alpha_t)ww_tdxdt\\
  	&=-\frac{1}{2}\int_Q(s^2|\nabla\alpha|_A^2+s\Delta_A\alpha+s\alpha_t)|w|^2_tdxdt=\frac{1}{2}\int_Q(s^2|\nabla\alpha|_A^2+s\Delta_A\alpha+s\alpha_t)_t|w|^2dxdt\\
  	&=\int_Q\lotw dxdt.
  	\end{aligned}
  	\end{equation}
  	
  	\begin{equation}\label{t5}
  	\begin{aligned}
  	T_5(Q)&=-2s\int_Q\langle\nabla\alpha,\nabla w\rangle_A(s^2|\nabla\alpha|_A^2+s\Delta_A\alpha+s\alpha_t)wdxdt\\
  	&=-s^2\int_Q\langle\nabla\alpha,\nabla |w|^2\rangle_A(s|\nabla\alpha|_A^2+\alpha_t) dxdt \\ 
  	&-\int_Q2s^2\lambda^3\varphi^2\langle\nabla\psi,\nabla w\rangle_A\modA{\nabla\psi}^2w+\lotwnw dxdt \\
  	&=-s^2\int_Q \dvA[{|w|^2(s|\nabla\alpha|_A^2+\alpha_t)\nabla\alpha}]dxdt\\
  	&+s^2\int_Q |w|^2\dvA[{(s|\nabla\alpha|_A^2+\alpha_t)\nabla\alpha}]dxdt\\
  	&-\int_Q2s^2\lambda^3\varphi^2\langle\nabla\psi,\nabla w\rangle_A\modA{\nabla\psi}^2w+\lotwnw dxdt \\
  	&=-s^2\int_\Sigma {|w|^2(s|\nabla\alpha|_A^2+\alpha_t)\dnuA{\alpha}}d\sigma dt\\
  	&+s^3\int_Q |w|^2\dvA{[| \nabla\alpha|_A^2\nabla\alpha]} dxdt+\int_Q\lotw dxdt\\
  	&-\int_Q2s^2\lambda^3\varphi^2\langle\nabla\psi,\nabla w\rangle_A\modA{\nabla\psi}^2w+\lotwnw dxdt \\
  	&=-s^2\int_\Sigma {|w|^2(s|\nabla\alpha|_A^2+\alpha_t)\dnuA{\alpha}}d\sigma dt+3\int_Q s^3\lambda^4\modA{\nabla\psi}^4\varphi^3|w|^2dxdt\\
  	&-\int_Q2s^2\lambda^3\varphi^2\langle\nabla\psi,\nabla w\rangle_A\modA{\nabla\psi}^2wdxdt+\int_Q\lotwnw dxdt.
  	\end{aligned}
  	\end{equation}
  	
  	\begin{equation}\label{t6}
  	\begin{aligned}
  	T_6(Q)&=-2s\int_Q\Delta_A\alpha(s^2|\nabla\alpha|_A^2+s\Delta_A\alpha+s\alpha_t)|w|^2dxdt\\
  	&=-2\int_Q s^3\lambda^4\varphi^3|\nabla\psi|^4_A|w|^2dxdt+ \int_Q \lotw dxdt.
  	\end{aligned}
  	\end{equation}
  	
  	Plugging all these computations in \eqref{estfi} we find
\begin{equation}\label{estfi1}
\begin{aligned}\displaystyle
&\int_Q s^3\lambda^4\modA{\nabla\psi}^4\varphi^3|w|^2  dx dt+\int_Q s\lambda^2 \varphi\modA{\nabla\psi}^2|\nabla w|_A^2dxdt\\
& +\int_Q2 s\lambda^2\varphi|\langle\nabla w, \nabla\psi\rangle_A|^2   dxdt-\int_Q2s^2\lambda^3\varphi^2\langle\nabla\psi,\nabla w\rangle_A\modA{\nabla\psi}^2wdxdt\\
&+\int_Q\lotwnw dxdt\\
&\le \int_Q \bar g e^{2s\alpha}dxdt-\mathcal{B}, 
\end{aligned}
\end{equation}  	
where $\mathcal B$ contains only boundary integrals:
$$
\begin{aligned}
\mathcal B=&-\int_\Sigma\frac{\partial w}{\partial \nu_A}w_td\sigma dt +\left[s\int_\Sigma|\nabla w|_A^2\frac{\partial\alpha}{\partial \nu_A}d\sigma dt-2s\int_\Sigma\frac{\partial w}{\partial \nu_A}\langle\nabla\alpha,\nabla w\rangle_Ad\sigma dt\right]\\
& -2s\int_\Sigma\frac{\partial w}{\partial \nu_A}\Delta_A\alpha wd\sigma dt-s^2\int_\Sigma {|w|^2(s|\nabla\alpha|_A^2+\alpha_t)\dnuA{\alpha}}d\sigma dt\\
&=S_1(\Sigma)+\ldots +S_4(\Sigma)=\mathcal{B}_0+\mathcal{B}_1,
\end{aligned}$$
where we denoted by $\mathcal{B}_0=\sum_{i=1}^4S_i(\Sigma_0),\mathcal{B}_1=\sum_{i=1}^4S_i(\Sigma_1)$  the surface integrals on $\Sigma_0$ respectively on $\Sigma_1$.

Observe that 
$$
 \frac34s^3\lambda^4\modA{\nabla\psi}^4\varphi^3|w|^2   +2 s\lambda^2\varphi|\langle\nabla w, \nabla\psi\rangle_A|^2   -2s^2\lambda^3\varphi^2\langle\nabla\psi,\nabla w\rangle_A\modA{\nabla\psi}^2w>0
$$
and, as $|\nabla\psi|>0$ in $\overline\Omega$, by absorbing the lower order terms in the principal terms we may conclude from \eqref{estfi1} that 
\begin{equation}\label{estfin2}
\begin{aligned}
&\int_Q s^3\lambda^4\varphi^3\modA{\nabla\psi}^4| w|^2+ s\lambda^2\varphi\modA{\nabla\psi}^2|\nabla w|_A^2  dxdt
&\le C\left[\int_Q \bar g e^{2s\alpha}dxdt- \mathcal{ B}\right]
\end{aligned}
\end{equation}
for some fixed constant $C>0$ and all $s,\lambda>0$ big enough.

We denote now by  $\tilde w:=ye^{s\tilde\alpha}$ and we consider the corresponding parabolic  problem satisfied by $\tilde{w}$ which is similar to \eqref{pbw}: 
 
\begin{equation}\label{pbtildaw}\left\lbrace
	\begin{aligned}
		&\tilde w_t-\Delta_A\tilde w+2s\langle\nabla\tilde\alpha,\nabla \tilde w\rangle_A-(s^2|\nabla\tilde\alpha|_A^2-s\Delta_A\tilde\alpha+s\tilde\alpha_t)\tilde w=\g e^{s\tilde \alpha} &\, \text{ in }(0,T)\times\Omega\\
		&\frac{\partial \w}{\partial \nu_A}+\left(\eta(x)-\lambda\tilde\varphi\frac{\partial\tilde\psi}{\partial \nu_A}\right)\w=0  &\, \text{on }(0,T)\times\partial\Omega
	\end{aligned}\right.
\end{equation}
The same procedure as in the case of $w$ provides the following estimate: 
\begin{equation}\label{estfin3}
	\begin{aligned}
		&\int_Q s^3\lambda^4\tilde\varphi^3\modA{\nabla\tilde \psi}^4|\tilde w|^2+ s\lambda^2 \tilde\varphi\modA{\nabla\tilde\psi}^2|\nabla \tilde w|_A^2  dxdt
		&\le C\left[\int_Q \bar g e^{2s\tilde\alpha}dxdt-\mathcal{\tilde B}\right]
	\end{aligned}
\end{equation}
for some fixed constant $C>0$ and all $s,\lambda>0$ big enough, where $\mathcal{\tilde B}$ contains the corresponding boundary integrals:
$$
\begin{aligned}
	\mathcal {\tilde B}=&-\int_\Sigma\frac{\partial \tilde w}{\partial \nu_A}\tilde w_td\sigma dt+\left[s\int_\Sigma|\nabla \tilde w|_A^2\frac{\partial\tilde \alpha}{\partial \nu_A}d\sigma dt-2s\int_\Sigma\frac{\partial \tilde w}{\partial \nu_A}\langle\nabla\tilde \alpha,\nabla \tilde w\rangle_Ad\sigma dt\right]\\
	& -2s\int_\Sigma\frac{\partial \tilde w}{\partial \nu_A}\Delta_A\tilde \alpha \tilde wd\sigma dt\\
	&-s^2\int_\Sigma {|\tilde w|^2(s|\nabla\tilde \alpha|_A^2+\tilde \alpha_t)\dnuA{\tilde \alpha}}d\sigma dt\\
 &=\tilde S_1(\Sigma)+\ldots +\tilde S_4(\Sigma)=\mathcal{\tilde B}_0+\mathcal{\tilde B}_1,
\end{aligned}$$
where we denoted by $\mathcal{\tilde B}_0=\sum_{i=1}^4\tilde S_i(\Sigma_0),\mathcal{\tilde B}_1=\sum_{i=1}^4\tilde S_i(\Sigma_1)$   the surface integrals on $\Sigma_0$ respectively on $\Sigma_1$.

Adding \eqref{estfin2}, \eqref{estfin3} and considering that $w=ye^{s\alpha},\tilde w=ye^{s\tilde \alpha} $ and $\psi\ge\tilde\psi$ we obtain the following estimate for $y$ for some other constant $C>0$ and $s,\lambda>0$ big enough: 
\begin{equation}
\begin{aligned}\label{estfi4}
&\int_Q [s^3\lambda^4\varphi^3 |y|^2+ s\lambda^2\varphi|\nabla y|_A^2 ] e^{2s\alpha}dxdt
&\le C\left[\int_Q \bar g e^{2s\alpha}dxdt- \mathcal{ B}-\mathcal{ \tilde B}\right]
\end{aligned}
\end{equation}
	
  At this point we will focus on the boundary terms, keeping in mind the following properties on $\Sigma_0, \Sigma_1$:

 \begin{equation}\label{psiuri}
 \begin{aligned}
 &\bullet \nabla \psi=-\nabla \tilde\psi, \Delta_A \psi=-\Delta_A\psi \text{ in  }\Omega\\
 &\bullet\,\nabla\psi|_{\Gamma_0}=c(x)\nu(x), \nabla\tilde\psi|_{\Gamma_0}=-c(x)\nu(x)\quad \text{ with } c(x)<0, x\in\Gamma_0;\\
&\bullet\,\dfrac{\partial\psi}{\partial \nu_A}|_{\Gamma_0}< 0, \quad \dfrac{\partial\tilde{\psi}}{\partial \nu_A}|_{\Gamma_0}=-\dfrac{\partial\psi}{\partial \nu_A}|_{\Gamma_0}> 0;\\
&\bullet\,\varphi=\tilde{\varphi}, \alpha=\tilde{\alpha} \text{ on }\Sigma_0;\\
 &\bullet\, \nabla\psi|_{\Gamma_1}=c(x)\nu(x),\nabla\tilde\psi|_{\Gamma_1}=-c(x)\nu(x) \quad \text{ with } c(x)>0, x\in\Gamma_1;\\
 &\bullet\, \dfrac{\partial\psi}{\partial \nu_A}|_{\Gamma_1}> 0, \quad \dfrac{\partial\tilde{\psi}}{\partial \nu_A}|_{\Gamma_1}=-\dfrac{\partial\psi}{\partial \nu_A}|_{\Gamma_1}< 0;\\
&\bullet\,\varphi>\tilde{\varphi}, \alpha>\tilde{\alpha} \text{ on }\Sigma_1.
 \end{aligned}
 \end{equation}
  
  The homogeneous boundary condition $\dfrac{\partial y}{\partial \nu_A}+\eta(x)y=0$ holds on the entire  $\Sigma$ and we compute and estimate now $\mathcal{B}+\mathcal{ \tilde B}$ by computing each sum of the form $S_i(\Sigma_0)+\tilde S_i(\Sigma_0)$ and, respectively, $S_i(\Sigma_1)+\tilde S_i(\Sigma_1)$.
  
  \smallskip
  
\noindent$\bullet$ $S_1(\Sigma)+\tilde S_1(\Sigma)$:
  \begin{equation}\label{t1sigma}
  \begin{aligned}
  S_1(\Sigma)&=-\int\limits_{\Sigma}w_t\frac{\partial w}{\partial \nu_A}d\sigma dt=-\int\limits_{\Sigma}(y_t+s\alpha_t y)\left(\frac{\partial y}{\partial \nu_A}+s\lambda\varphi y\frac{\partial \psi}{\partial \nu_A}\right)e^{2s\alpha}d\sigma dt\\
  &=-\int\limits_{\Sigma}(y_t+s\alpha_t y)\left(-\eta(x)y+s\lambda\varphi y\frac{\partial \psi}{\partial \nu_A}\right)e^{2s\alpha}d\sigma dt\\
  &=-\int\limits_{\Sigma}\left(\frac{|y|^2}{2}\right)_t\left(-\eta(x)+s\lambda\frac{\partial \psi}{\partial \nu_A}\varphi\right)e^{2s\alpha}d\sigma dt-\int_\Sigma |y|^2s\alpha_t \left(-\eta(x)+s\lambda\frac{\partial \psi}{\partial \nu_A}\varphi\right)e^{2s\alpha}d\sigma dt\\
  &=\int\limits_{\Sigma}\frac{|y|^2}{2}\left[s\lambda\frac{\partial \psi}{\partial \nu_A}\varphi_t +2s\alpha_t\left(-\eta(x)+s\lambda\frac{\partial \psi}{\partial \nu_A}\varphi\right)\right]e^{2s\alpha}d\sigma dt\\&
  -\int_\Sigma |y|^2s\alpha_t \left(-\eta(x)+s\lambda\frac{\partial \psi}{\partial \nu_A}\varphi\right)e^{2s\alpha}d\sigma dt\\
&=\int\limits_{\Sigma}\frac{|y|^2}{2}s\lambda\frac{\partial \psi}{\partial \nu_A}\varphi_te^{2s\alpha}d\sigma d t
      \end{aligned}
  \end{equation}
  Correspondingly,
  \begin{equation}\label{t1tsigma}
  \tilde S_1(\Sigma)=\int\limits_{\Sigma}\frac{|y|^2}{2}s\lambda\frac{\partial\tilde \psi}{\partial \nu_A}\tilde \varphi_te^{2s\tilde \alpha}d\sigma d t
  \end{equation}
  Considering \eqref{psiuri} we have
  \begin{equation}\label{s01}
  S_1(\Sigma_0)+\tilde S_1(\Sigma_0)=0
  \end{equation}
  and
  \begin{equation}\label{s11}
 | S_1(\Sigma_1)+\tilde S_1(\Sigma_1)|\le C\int\limits_{\Sigma_1}s\lambda\varphi^2e^{2s\alpha}|\zeta(y)|^2d\sigma dt.
  \end{equation}
  
  \noindent$\bullet$ $S_2(\Sigma)+\tilde S_2(\Sigma)$:
  \begin{equation}\label{t2sigma}
  \begin{aligned}
  S_2(\Sigma)&=s\int\limits_\Sigma\frac{\partial\alpha}{\partial \nu_A}|\nabla w|_A^2d\sigma dt-2s\int\limits_\Sigma\frac{\partial w}{\partial \nu_A}\langle\nabla \alpha,\nabla w\rangle_Ad\sigma dt\\
 &=\int\limits_\Sigma s\lambda\varphi|\nabla w|_A^2\frac{\partial\psi}{\partial \nu_A}d\sigma dt-2\int\limits_\Sigma s\lambda\varphi\langle\nabla \psi,\nabla w\rangle_A \frac{\partial w}{\partial \nu_A}d\sigma dt\\ 
 &=\int\limits_\Sigma s\lambda\varphi(|\nabla y|_A^2+s^2\lambda^2\varphi^2|\nabla \psi|_A^2y^2+2s\lambda\varphi\langle\nabla \psi,\nabla y\rangle_A y)\frac{\partial\psi}{\partial \nu_A}e^{2s\alpha}d\sigma dt\\
 &-2\int\limits_\Sigma s\lambda\varphi (\langle\nabla \psi,\nabla y\rangle_A+s\lambda\varphi|\nabla \psi|_A^2y)\left( \frac{\partial y}{\partial \nu_A}+s\lambda\varphi\frac{\partial\psi}{\partial \nu_A}y\right)e^{2s\alpha}d\sigma dt\\
 &=\int\limits_\Sigma s\lambda\varphi|\nabla y|_A^2\frac{\partial\psi}{\partial \nu_A}e^{2s\alpha}d\sigma dt-2\int\limits_\Sigma s\lambda\varphi\langle\nabla \psi,\nabla y\rangle_A\frac{\partial y}{\partial \nu_A}e^{2s\alpha}d\sigma dt\\
 &-2\int\limits_\Sigma s^2\lambda^2\varphi^2|\nabla \psi|_A^2y\frac{\partial y}{\partial \nu_A}e^{2s\alpha}d\sigma dt-\int\limits_\Sigma s^3\lambda^3\varphi^3|\nabla \psi|_A^2y^2\frac{\partial \psi}{\partial \nu_A}e^{2s\alpha}d\sigma dt\\
 &=\int\limits_\Sigma s\lambda\varphi c(x)\modA\nu^2|\nabla y|_A^2e^{2s\alpha}d\sigma dt-2\int\limits_\Sigma s\lambda\varphi c(x)\left| \dnuA{y}\right|^2e^{2s\alpha}d\sigma dt\\
  &-2\int\limits_\Sigma s^2\lambda^2\varphi^2 c(x)^2\modA{\nu}^2y\frac{\partial y}{\partial \nu_A}e^{2s\alpha}d\sigma dt-\int\limits_\Sigma s^3\lambda^3\varphi^3c(x)^3\modA{\nu}^4|y|^2e^{2s\alpha}d\sigma dt.\\
  \end{aligned}
  \end{equation}
  Correspondingly,
  \begin{equation}\label{t2tsigma}
    \begin{aligned}
    \tilde S_2(\Sigma)&=\int\limits_\Sigma s\lambda\tilde \varphi|\nabla y|_A^2\frac{\partial\tilde \psi}{\partial \nu_A}e^{2s\tilde \alpha}d\sigma dt-2\int\limits_\Sigma s\lambda\tilde \varphi\langle\nabla\tilde  \psi,\nabla y\rangle_A\frac{\partial y}{\partial \nu_A}e^{2s\tilde \alpha}d\sigma dt\\
   &-2\int\limits_\Sigma s^2\lambda^2\tilde \varphi^2|\nabla\tilde  \psi|_A^2y\frac{\partial y}{\partial \nu_A}e^{2s\tilde \alpha}d\sigma dt-\int\limits_\Sigma s^3\lambda^3\tilde \varphi^3|\nabla \tilde \psi|_A^2y^2\frac{\partial\tilde  \psi}{\partial \nu_A}e^{2s\tilde \alpha}d\sigma dt\\
    &=-\int\limits_\Sigma s\lambda\tilde \varphi c(x)\modA\nu^2|\nabla y|_A^2e^{2s\tilde \alpha}d\sigma dt+2\int\limits_\Sigma s\lambda\tilde \varphi c(x)\left| \dnuA{y}\right|^2e^{2s\tilde \alpha}d\sigma dt\\
     &-2\int\limits_\Sigma s^2\lambda^2\tilde \varphi^2 c(x)^2\modA{\nu}^2y\frac{\partial y}{\partial \nu_A}e^{2s\tilde \alpha}d\sigma dt+\int\limits_\Sigma s^3\lambda^3\tilde \varphi^3c(x)^3\modA{\nu}^4|y|^2e^{2s\tilde \alpha}d\sigma dt.
    \end{aligned}
    \end{equation}
  Considering \eqref{psiuri} we have
  \begin{equation}\label{s02}\begin{aligned}
  S_2(\Sigma_0)+\tilde S_2(\Sigma_0)&=-4\int\limits_{\Sigma_0 }s^2\lambda^2\varphi^2 c(x)^2\modA{\nu}^2y\frac{\partial y}{\partial \nu_A}e^{2s\alpha}d\sigma dt\\&=4\int\limits_{\Sigma_0} s^2\lambda^2\varphi^2 c(x)^2\modA{\nu}^2\eta(x)|y|^2e^{2s\alpha}d\sigma dt,
  \end{aligned}
  \end{equation}
  and
  \begin{equation}\label{s12}
  |S_2(\Sigma_1)+\tilde S_2(\Sigma_1)|\le C \int\limits_{\Sigma_1} s^3\lambda^3\varphi^3|\zeta(y)|^2e^{2s\alpha}d\sigma dt.
  \end{equation}

  \noindent$\bullet$ $S_3(\Sigma)+\tilde S_3(\Sigma)$:
  \begin{equation}
  \label{t3sigma}
  \begin{aligned}
  S_3(\Sigma)&=-2s\int\limits_{\Sigma} \Delta_A\alpha\frac{\partial w}{\partial \nu_A}wd\sigma dt\\
   &=-2s\int\limits_{\Sigma} (\lambda^2\varphi|\nabla\psi|^2_A+\lambda\varphi\Delta_A\psi)\left(\frac{\partial y}{\partial \nu_A}+s\lambda\varphi y\frac{\partial\psi}{\partial \nu_A}\right)ye^{2s\alpha}d\sigma dt\\
     &=-2\int\limits_{\Sigma} (s\lambda^2\varphi|\nabla\psi|^2_A+s\lambda\varphi\Delta_A\psi)\frac{\partial y}{\partial \nu_A}ye^{2s\alpha}d\sigma dt-\\
      &-2\int\limits_{\Sigma} (s^2\lambda^3\varphi^2|\nabla\psi|^2_A+s^2\lambda^2\varphi^2\Delta_A\psi) \frac{\partial\psi}{\partial \nu_A}|y|^2e^{2s\alpha}d\sigma dt\\
      &=2\int\limits_{\Sigma} (s\lambda^2\varphi c(x)^2\modA{\nu}^2+s\lambda\varphi\Delta_A\psi)\eta(x)|y|^2e^{2s\alpha}d\sigma dt-\\
            &-2\int\limits_{\Sigma} (s^2\lambda^3\varphi^2c(x)^2\modA{\nu}^2+s^2\lambda^2\varphi^2\Delta_A\psi) c(x)\modA{\nu}^2|y|^2e^{2s\alpha}d\sigma dt.\\
  \end{aligned}
  \end{equation}
  Correspondingly,
  \begin{equation}
  \label{t3tsigma}
  \begin{aligned}
  \tilde S_3(\Sigma)&=-2\int\limits_{\Sigma} (s\lambda^2\tilde \varphi|\nabla\tilde \psi|^2_A+s\lambda\tilde \varphi\Delta_A\tilde \psi)\frac{\partial y}{\partial \nu_A}ye^{2s\tilde \alpha}d\sigma dt\\
        &-2\int\limits_{\Sigma} (s^2\lambda^3\tilde \varphi^2|\nabla\tilde \psi|^2_A+s^2\lambda^2\tilde \varphi^2\Delta_A\tilde \psi) \frac{\partial\tilde \psi}{\partial \nu_A}|y|^2e^{2s\tilde \alpha}d\sigma dt\\
        &=2\int\limits_{\Sigma} (s\lambda^2\tilde \varphi c(x)^2\modA{\nu}^2-s\lambda\tilde \varphi\Delta_A \psi)\eta(x)|y|^2e^{2s\alpha}d\sigma dt\\
              &+2\int\limits_{\Sigma} (s^2\lambda^3\tilde \varphi^2c(x)^2\modA{\nu}^2-s^2\lambda^2\tilde \varphi^2\Delta_A\psi) c(x)\modA{\nu}^2|y|^2e^{2s\alpha}d\sigma dt\\
  \end{aligned}
  \end{equation}
  Considering \eqref{psiuri} we find
  \begin{equation}
  \begin{aligned}\label{s03}
  S_3(\Sigma_0)+\tilde S_3(\Sigma_0)&= 4\int\limits_{\Sigma_0} s\lambda^2\tilde \varphi c(x)^2\modA{\nu}^2\eta(x)|y|^2e^{2s\alpha}d\sigma dt\\
  &  -4\int\limits_{\Sigma_0}  s^2\lambda^2\varphi^2\Delta_A\psi c(x)\modA{\nu}^2|y|^2e^{2s\alpha}d\sigma dt
  \end{aligned}
  \end{equation}
  and 
   \begin{equation}\label{s13}
    |S_3(\Sigma_1)+\tilde S_3(\Sigma_1)|\le C \int\limits_{\Sigma_1} s^2\lambda^3\varphi^2|\zeta(y)|^2e^{2s\alpha}d\sigma dt.
    \end{equation}

  \noindent$\bullet$ $S_4(\Sigma)+\tilde S_4(\Sigma)$:

  \begin{equation}\label{t4sigma}
  \begin{aligned}
  S_4(\Sigma)&=-s^2\int_{\Sigma}(s|\nabla\alpha|_A^2 +\alpha_t)\lambda\varphi\frac{\partial\psi}{\partial \nu_A}|w|^2d\sigma dt\\
  &=-\int_{\Sigma}\left[s^3\lambda^3\varphi^3|\nabla\psi|_A^2  +s^2\lambda\alpha_t\varphi \right]|y|^2e^{2s\alpha}c(x)\modA{\nu}^2d\sigma dt.
  \end{aligned}
  \end{equation}
  
  Correspondingly,
  \begin{equation}\label{t4tsigma}
    \begin{aligned}
   \tilde  S_4(\Sigma)&=-s^2\int_{\Sigma}(s|\nabla\tilde \alpha|_A^2 +\tilde \alpha_t)\lambda\tilde \varphi\frac{\partial\tilde \psi}{\partial \nu_A}|w|^2d\sigma dt\\
    &=\int_{\Sigma}\left[s^3\lambda^3\tilde \varphi^3|\nabla\psi|_A^2  +s^2\lambda\tilde \alpha_t\tilde \varphi \right]|y|^2e^{2s\tilde \alpha}c(x)\modA{\nu}^2d\sigma dt.
    \end{aligned}
    \end{equation}
  Taking into account \eqref{psiuri} we find
  
  \begin{equation}
    \begin{aligned}\label{s04}
    S_4(\Sigma_0)+\tilde S_4(\Sigma_0)&= 0\end{aligned}
    \end{equation}
    and 
     \begin{equation}\label{s14}
      |S_4(\Sigma_1)+\tilde S_4(\Sigma_1)|\le C \int\limits_{\Sigma_1} s^3\lambda^3\varphi^3|\zeta(y)|^2e^{2s\alpha}d\sigma dt.
      \end{equation}
      
   Plugging \eqref{s01},\eqref{s11},\eqref{s02},\eqref{s12},\eqref{s03},\eqref{s13},\eqref{s04},\eqref{s14} into \eqref{estfi4} and considering also that $\Delta_A\psi>0$ in $\overline\Omega$ (by Remark \ref{remsubharmpsi}) we find that
    
    \begin{equation}
    \begin{aligned}\label{estfi5}
    &\int_Q [s^3\lambda^4\varphi^3 |y|^2+ s\lambda^2\varphi|\nabla y|_A^2 ] e^{2s\alpha}dxdt+ \int\limits_{\Sigma_0}  s^2\lambda^2\varphi^2 |y|^2e^{2s\alpha}d\sigma dt\\
    &\le C\left[\int_Q \bar g e^{2s\alpha}dxdt+\int\limits_{\Sigma_1}s^3\lambda^3\varphi^3|\zeta(y)|^2e^{2s\alpha}d\sigma dt\right],
    \end{aligned}
    \end{equation}
    which concludes the proof.
  \end{proof}
  
 \begin{remark}
 One may easily obtain Carleman estimates with boundary observations for the general system \eqref{sysinitial}, by combining Carleman estimates established in Proposition \ref{lemaCarlemanclassic} for each equation, with the usual mechanism of  absorbing the lower order terms in the free terms and applying Cauchy inequality.
 \end{remark}

\section{Source stability estimates}\label{sec:stability}
Our approach to obtain source stability estimates uses, through an argument by contradiction,  the following auxiliary result based essentially on Carleman estimates with boundary observations.

Consider a family of problems of type \eqref{pbparabmax}

\begin{equation}\label{pbparabmax-m}
     \begin{cases}
     D_ty+ {L}^{(m)}y=g  &\text{ in }Q\\
     \displaystyle\beta(x)\frac{\partial y}{\partial \nu_{A}}+\eta(x)y=0  &\text{ on }(0,T)\times\partial\Omega,
     \end{cases}
     \end{equation} with corresponding elliptic parts   
     \begin{equation}\label{opparabdiv-m}
        L^(m)u=-\sum_{j,k=1}^N D_j(a^{jk}(x) D_ku)+\sum\limits_{k=1}^Nb^{k}(x) D_ku+c^{(m)}(t,x)u,
        \end{equation}
        and coefficients satisfying the same assumptions as operator \eqref{opparabdiv2}
\begin{lemma}\label{convergente} Assume that the zero order coefficients $c^{(m)}$ satisfy an uniform $L^\infty$ bound: there exists $M>0$ such that
$$
\|c^{(m)}\|_{L^\infty(Q)}\le M.
$$
Consider $(y^{(m)})_m$ a   sequence of (variational) solutions for problem \eqref{pbparabmax-m} with corresponding sources $(g^{(m)})_m\subset L^2(Q)$. If   $(g^{(m)})_m$ is bounded in $L^2(Q)$  and $(\zeta(y^{(m)}))_m$ is bounded in $L^2(\Sigma_1)$, then there exist subsequences also denoted by $(y^{(m)})_m, (g^{(m)})_m$  
and $y\in L^2_{loc}(0,T; H^{1}(\Omega))$, $g\in L^2(Q)$ such that   for $\epsilon>0$ small  
\begin{eqnarray}
\label{c1}c^{(m)} \rightharpoonup c,\text{ weakly-* in } L^\infty(Q);\\
\label{c2}g^{(m)} \rightharpoonup g\,\text{ weakly in } L^2(\epsilon, T-\epsilon; L^2(\Omega));\\
\label{c3}y^{(m)} \rightharpoonup y\,\text{ weakly in } L^2(\epsilon, T-\epsilon; H^{1}(\Omega));\\
\label{c4}y^{(m)}\rightarrow y \,\text{ strongly in } L^2(\epsilon, T-\epsilon; L^{2}(\Omega)).
\end{eqnarray}
Moreover, $y$ is a variational solution corresponding to the source $g$.
\end{lemma}
\begin{proof}
Considering the functional framework given by the Hilbert spaces introduced in \eqref{Vspace} consider $a:V\times V\rightarrow \R$, 
 \begin{equation}\label{a-form-m}
   a(u,v)=\int_\Omega\langle A(x)\nabla u, \nabla v\rangle+\langle\textbf{b}, \nabla u\rangle v dx+\int_{\Gamma_R}\eta(x) u v d\sigma.
   \end{equation}
    The variational solutions $y^{(m)}$ corresponding to sources $g^{(m)}$ belong to $C([0,T];H)\cap L^2(0,T;V)$ and satisfy
  \begin{eqnarray}
  \langle y^{(m)}(t_2),v\rangle_H-\langle y^{(m)}(t_1),v\rangle_H+\int_{t_1}^{t_2} a(  y^{(m)}(\tau),v)+\langle c^{(m)}(\tau,\cdot)y^{(m)}(\tau),v\rangle_H d\tau\nonumber\\
  =\int_{t_1}^{t_2}\langle g^{(m)}(\tau),v\rangle_Hd\tau,\quad\forall 0<t_1<t_2<T,\forall v\in V.\label{varsol-m}
  \end{eqnarray}
  The boundedness of $(g^{(m)})_m$ in $L^2(Q)$  and of $c^{(m)}$ in $L^\infty(Q)$ allows to extract subsequences such that \eqref{c1} and \eqref{c2} hold.
  
The boundedness of $(g^{(m)})_m$ in $L^2(Q)$ and of $(\zeta(y^{(m)}))_m$ in $L^2(\Sigma_1)$ assures, by Carleman estimate for the solution $y^{(m)}$ of \eqref{pbparabmax-m} corresponding to the sources $g^{(m)}$, that $(y^{(m)})_m$ is bounded in
$L^2_{loc}(0,T;H^1(\Omega))$ which is in fact boundedness in $L^2_{loc}(0,T;V)$, meaning boundedness  in $L^2(\epsilon, T-\epsilon;V),\,\forall 0<\epsilon<T/2$. So, we may further extract subsequences such that \eqref{c3} holds.

Parabolic regularity results  ensure boundedness of $y^{(m)}$ in $L^2_{loc}(0,T; D(\A^\flat))$ and of $D_ty^{(m)}$ in  $L^2_{loc}(0,T; V^*)$. Consequently, by Aubin-Lions Lemma we may further extract  subsequences such that \eqref{c4} is satisfied.

We have now all necessary convergences to pass to the limit in \eqref{varsol-m} and conclude that $y$ is a variational solution corresponding to the source $g$.
 
\end{proof}

\subsubsection*{Linear systems. Proof of Theorem \ref{thlinear}.}

We focus now on proving  source estimates for the linear parabolic systems \eqref{sysinitial} with sources $\textbf{g}$ belonging to  $\mathcal{G}_{k}$.
We have to prove that for $k,M>0$ there exists $C=C(k,M)$ such that for $\textbf{g}\in\mathcal{G}_{k}$, and zero-order coefficients $\textbf{c}=(c_{il})_{il}$ satisfying $\|\textbf{c}\|_{L^\infty(Q)}\le M$ there exists $C=C(k,M)$ such that 
\begin{equation}
\begin{aligned}
&\left\|\textbf{g}\right\|_{L^2(Q)}\leq C\|\zeta(\textbf{y})\|_{L^2(\Sigma_1)}.
\end{aligned}
\end{equation}
 We argue by contradiction. There exist thus $k,M>0$ and   sequences $(\textbf{g}^{(m)})_m\subset \mathcal{G}_{k} $,  zero order coefficients $\textbf{c}^{(m)}=(c^{(m)}_{il})_{il}$ satisfying $\|\textbf{c}^{(m)}\|_{L^\infty(Q)}\le M$ and  solutions $y^{(m)}$  corresponding to the sources and system \eqref{sysinitial} (with $\textbf{c}^{(m)}$ as zero order coefficients) such that 
 \begin{equation}\label{contrgzeta}
\|\textbf{g}^{(m)}\|_{L^2(Q)}>m\|\zeta(\textbf{y}^{(m)})\|_{L^2(\Sigma_1)}.
\end{equation}
With no loss of generality we may suppose that $\left\|\textbf{g}^{(m)}\right\|_{L^2(Q)}=1$. Apply now Lemma \ref{convergente} addapted to systems, extract  subsequences such that for all $0<\epsilon<T/2$, similar convergences to \eqref{c1}-\eqref{c4} hold:
\begin{eqnarray}
\label{c1s}\textbf{c}^{(m)} \rightharpoonup \textbf{c},\text{ weakly-* in } [L^\infty(Q)]^n;\\
\label{c2s}\textbf{g}^{(m)} \rightharpoonup \textbf{g}\,\text{ weakly in } L^2(0, T; [L^2(\Omega)]^n);\\
\label{c3s}\textbf{y}^{(m)} \rightharpoonup \textbf{y}\,\text{ weakly in } L^2(\epsilon, T-\epsilon; [H^{1}(\Omega)]^n);\\
\label{c4s}\textbf{y}^{(m)}\rightarrow \textbf{{y}} \,\text{ strongly in } L^2(\epsilon, T-\epsilon; [L^{2}(\Omega)]^n).
\end{eqnarray}
Moreover, $\textbf{y}$ is variational solution to \eqref{sysinitial} with source $\textbf{g}$ and zero-order coefficients $\textbf{c}=(c_{il})_{il}$.

By \eqref{contrgzeta} we obtain that 
$$
\mathbf{\zeta}_i(\textbf{y})=0.\forall i,
$$
and by the compatibility conditions in (H4), since $\textbf{y}\in L^2(0,T;[H^2(\Omega)]^n)$ we find that 
\begin{equation}\label{posext}
y_i(t,\cdot)=0\text{ and }\frac\partial{\partial\nu}y_i(t,\cdot)=0 \text{ on } \Gamma_1,\, a.e. t\in (0,T).
\end{equation}

We observe now that the weak limit $\textbf{g}$ of $(\textbf{g}^{(m)})_m$ is not zero. Indeed, by hypothesis $\textbf{g}^{(m)}\in\mathcal{G}_{k}$ and thus, as  $1\in L^{2}(Q)$, $g_i\ge0$ and weak convergence \eqref{c2s} we have
 $$1=\|\textbf{g}^{(m)}\|_{L^2(Q)}\le k\|\textbf{g}^{(m)}\|_{L^1(Q)}=k\sum_i\int_{Q}g_i^{(m)}\rightarrow\sum_i\int_{Q}g_i =k\|\textbf{g}\|_{L^1(Q)}.$$
 Consequently,
 \begin{equation}\label{pozg}
 \textbf{g}\not\equiv 0 \text{ in } Q \text{ and }\textbf{g}\ge0.
 \end{equation}
Consider now  a slighter larger domain $\tilde\Omega$ extending $\Omega$ in the normal directon to $\Gamma_1$:

$$
\tilde\Omega=\Omega\cup\{x+\tau\nu(x)|x\in\Gamma_1,0\le\tau<\varepsilon\},
$$
for some $\varepsilon>0$ small enough such that $\tilde\Omega$ has smooth boundary,
$$
\partial\tilde\Omega=\Gamma_0\cup\tilde\Gamma_1,\,\tilde\Gamma_1=\{x+\varepsilon\nu(x)|x\in\Gamma_1\}
$$
Extend now the coefficients $a_i^{jk},b_i $ and $c_{il}$ of the operators $L_i$ to $\tilde\Omega$ and respectively $(0,T)\times\tilde\Omega$ to corresponding functions $\tilde a_i^{jk},\tilde b_i $ and $\tilde c_{il}$ such that the new coefficients satisfy the same hypotheses associated to problem \eqref{sysinitial} and also $\tilde c_{il}\le0$ in $(0,T)\times\tilde\Omega$. Consider now the parabolic system:

\begin{equation}\label{sysinitial-3}
   \left\lbrace
   \begin{array}{ll}
   D_tz_i-\sum\limits_{j,k=1}^N D_j(\tilde a_i^{jk} D_kz_i)+\sum\limits_{{\substack{k=\overline{1,N}}}}\tilde b_i^{k} (x)D_kz_i+\sum\limits_{l=\overline{1,n}}\tilde c_{il}(t,x) z_l=\chi_\Omega g_i, &\text{ in }(0,T)\times\Omega,
   \\
   \displaystyle\beta_i(x)\frac{\partial z_i}{\partial \nu_{A_i}}+\eta_i(x)z_i=0 ,  &\text{ on }(0,T)\times\Gamma_0,\\
   z_i=0 &\text{ on }(0,T)\times\tilde\Gamma_1
   \end{array}
   \right.\end{equation} 
 where we denoted by $\chi_\Omega w$ the extension with zero to $\tilde\Omega$ of a function $w$ defined on $\Omega$.

We observe now that $\eqref{posext}$ allows to say that $\textbf{z}=\chi_\Omega \textbf{y}$ is a variational solution to \eqref{sysinitial-3} and $\chi_\Omega \textbf{y}\equiv 0$ in $(0,T)\times(\tilde\Omega\setminus\Omega)$. But as observed in \eqref{pozg}, $\textbf{g}\ge0$ and $\textbf{g}\not\equiv0$ which means that $ \chi_\Omega\textbf{g}\not\equiv0$. But  this is a contradiction by positivity improving properties established in Theorem \ref{strmaxsys}.
\fin

\subsubsection*{Semilinear reaction-diffusion systems. Proof of Theorem \ref{thnonlinear}}
Consider first in  (H3)  the first assumption $f_i(t,x,y_1,\ldots,y_{i-1},0,y_{i+1},\dots, y_n)=0,  i =\overline{1,n},   t>0,x\in\overline\Omega,\, \textbf{y}\ge0.$ In this case, for  $\textbf{y}\in\mathcal{F}_M$ a solution to \eqref{rd-sys} we observe that, by a linearization mechanism,  $\textbf{y}$ is solution to the "linearized" system
\begin{equation}\label{rd-sys-lin}
\left\lbrace
\begin{array}{ll}
D_ty_i-\sum\limits_{j,k=1}^N D_j(a_i^{jk} D_ky_i)+\sum\limits_{k=1}^N b_i^k D_ky_i+c^{\textbf{y}}_i(t,x)y_i=g_i  &(0,T)\times\Omega,
\\
\displaystyle\beta_i(x)\frac{\partial y_i}{\partial \nu_{A_i}}+\eta_i(x)y_i=0  &(0,T)\times\partial\Omega,  \\

\end{array} i\in\overline{1,n}
\right.\end{equation} 
where 
$$
c^{\textbf{y}}_i(t,x)=\int_0^1\frac{\partial}{\partial y_i}f_i(t,x,y_1(t,x),\ldots,y_{i-1}(t,x),\tau y_i(t,x),y_{i+1}(t,x),\dots, y_n(t,x))d\tau,
$$
and $ c^{\textbf{y}}_i$ is uniformly bounded in $L^\infty(Q)$ and Theorem \ref{thlinear} applies.

Take now the second case of (H3), \textit{i.e.} $f_i(t,x,y_1,\ldots,y_{i-1},0,y_{i+1},\dots, y_n)\leq 0,  i =\overline{1,n},   t>0,x\in\overline\Omega,\,\textbf{ y}\ge0$
  and
  $$
  \frac{\partial}{\partial y_l}f_i(t,x,\textbf{y})\leq0, i,l =\overline{1,n}, i\ne l,  t>0,x\in\overline\Omega,\, \textbf{y}\ge0.
  $$
  By a similar linearization mechanism,  $\textbf{y}\in \mathcal{F}_M$ is solution to the "linearized" system
  \begin{equation}\label{rd-sys-linS}
  \left\lbrace
  \begin{array}{ll}
  D_ty_i-\sum\limits_{j,k=1}^N D_j(a_i^{jk} D_ky_i)+\sum\limits_{k=1}^N b_i^k D_ky_i+\sum_{l=1}^nc^{\textbf{y}}_{il}(t,x)y_l=g_i  &(0,T)\times\Omega,
  \\
  \displaystyle\beta_i(x)\frac{\partial y_i}{\partial \nu_{A_i}}+\eta_i(x)y_i=0  &(0,T)\times\partial\Omega,  \\
  
  \end{array} i\in\overline{1,n}
  \right.\end{equation} 
  where 
  $$
  c^{\textbf{y}}_{il}(t,x)=\int_0^1\frac{\partial}{\partial y_l}f_i(t,x,\tau y_1(t,x),\ldots, \tau y_n(t,x))d\tau.
  $$
 Observe that  $ c_{il}^{\textbf{y}}$ are uniformly bounded in $L^\infty(Q)$ and satisfy the sign conditions in order to apply Theorem \ref{thlinear}  and conclude the proof.\fin

\end{document}